\documentclass[11pt,letterpaper]{amsart}

\usepackage{amssymb,amsmath,amsthm,amstext,amscd,latexsym,graphics,graphicx,bbm,caption}
\usepackage[usenames,dvipsnames,svgnames,table]{xcolor}
\usepackage{hyperref}
\usepackage{tikz}
\usetikzlibrary{positioning}
\usepackage{tikz-cd}
\tikzset{>=stealth}
\hypersetup{plainpages=false,colorlinks=true, pagebackref}
\hypersetup{citecolor=Sepia,linkcolor=blue, urlcolor=blue}

\newtheorem{intro-thm}{Theorem}[]
\theoremstyle{plain}
\newtheorem{thm}{Theorem}[section]
\newtheorem{theorem}[thm]{Theorem}

\newtheorem{lemma}[thm]{Lemma}
\newtheorem{corollary}[thm]{Corollary}
\newtheorem{proposition}[thm]{Proposition}

\theoremstyle{definition}
\newtheorem{remark}[thm]{Remark}
\newtheorem{conjecture}[thm]{Conjecture}
\newtheorem{question}[thm]{Question}

\newtheorem{condition}[thm]{Condition}

\newtheorem{definition}[thm]{Definition}
\newtheorem{example}[thm]{Example}

\renewcommand{\tilde}{\widetilde}
\newcommand{\sX}{{\mathcal X}}

\input{xy}
\xyoption{all}

\title[$\mathbb{A}^1$-Connected components and characterisation of $\mathbb{A}^2$]{$\mathbb{A}^1$-Connected components and characterisation of $\mathbb{A}^2$}
\author{Utsav Choudhury}
\address{Indian Statistical Institute\\ Stat Math Unit\\203 B.T. Road \\ Kolkata - 700108\\ India}
\email{utsav@isical.ac.in}
\author{Biman Roy}

\address{Indian Statistical Institute\\ Stat Math Unit\\203 B.T. Road \\ Kolkata - 700108\\ India}
\email{bimanroy31@gmail.com}
\begin{document}
\maketitle
\begin{abstract}
In this article we prove that any $\mathbb{A}^1$-contractible smooth complex surface is isomorphic as a variety to $\mathbb{C}^2$. We show that the $\mathbb{A}^1$-connected component of a variety $X$ contains the information about $\mathbb{A}^1$-s in $X$.
\end{abstract}
\tableofcontents
\section{Introduction}
 A natural problem in affine algebraic geometry is to determine whether a given affine variety $X$ is isomorphic to the affine $n$-space $\mathbb{A}^n_k$ over a field $k$. Several invariants are used to answer this question. For instance, a complex affine variety isomorphic to $\mathbb{A}^n_{\mathbb{C}}$ is topologically contractible, has trivial Picard group and has trivial group of units. Ramanujam, in his fundamental work \cite{ra}, discovered that a non-singular complex algebraic surface is isomorphic to $\mathbb{A}^2_{\mathbb{C}}$ if and only if it is topologically contractible and simply connected at infinity. On the other hand, Miyanishi gave characterisation of $\mathbb{A}^2_{\mathbb{C}}$ using $\mathbb{G}_a$- action \cite[Theorem 1]{mi}. Thus, the invariants coming from the $\mathbb{G}_a$-actions, such as Makar-Limanov invariant as well as topological contractibility play an important role in algebraic characterisation of the affine spaces. This article arose from an attempt to relate these invariants.

We will use the bridge of $\mathbb{A}^1$-homotopy theory to establish a relation between algebra and topology. Let $k$ be a field and $Sm/k$ be the category of smooth finite type $k$-schemes. The unstable motivic homotopy category, denoted by $\mathbf{H}(k)$, is the category constructed by Morel-Voevodsky, where the homotopies are parametrized by the affine line $\mathbb{A}^1$ (see Appendix \ref{app} for the construction and relevant properties). Roughly speaking, $\mathbf{H}(k)$ is the most natural way of attaching simplicial homotopy and inverting the projection map  $pr : X \times \mathbb{A}^1 \to X$ for any smooth $k$-scheme $X$. Of course, during this process several other morphisms get inverted. A smooth $k$-scheme $X$ (more generally a space $\mathcal{X}$)  is called $\mathbb{A}^1$-contractible if the natural map $X \to Spec \ k$ (resp. $\mathcal{X} \to Spec \ k$) is an isomorphism in $\mathbf{H}(k)$. For example, $\mathbb{A}^n_k$-s are $\mathbb{A}^1$-contractible. Note that over $\mathbb{C}$, $\mathbb{A}^1$-contractibility of a smooth $\mathbb{C}$-scheme implies topological contractibility of the complex manifold $X(\mathbb{C})$. \cite[Lemma 2.5.]{ad}.\par

\subsection{$\mathbb{A}^1$-contractible smooth $k$-schemes}
Let us first consider the topological side. In dimension $1$ and $2$, the real line and the real plane are the only topologically contractible open manifolds. Whitehead first constructed topologically contractible open manifold of dimension $3$ which is not homeomorphic to $\mathbb{R}^3$ \cite{white}. In fact for every $n \geq 3$, there are infinitely many pairwise non-homeomorphic topologically contractible open manifolds of dimension $n$ (\cite{mcmillan}, \cite{mazur}, \cite{po}, \cite{curtis}, \cite{gl}).  

The base field $k$ is henceforth assumed to be of characteristic zero. In dimension $1$, $\mathbb{A}^1$ is the only $\mathbb{A}^1$-contractible smooth scheme \cite[Claim 5.7]{ad}. On the other hand, for every $n \geq 4$, Asok and Doran constructed infinitely many pairwise non-isomorphic strictly quasi-affine $\mathbb{A}^1$-contractible smooth schemes of dimension $n$, which are quotients of the affine spaces by a $\mathbb{G}_a$-action \cite[Theorem 5.1, 5.3]{ad}. In dimension $n \geq 6$, there are arbitrary dimensional moduli of pairwise non-isomorphic strictly quasi-affine $\mathbb{A}^1$-contractible smooth schemes of dimension $n$ arising from the quotient of an action of a unipotent group on an affine space \cite[Theorem 5.3]{ad}. However, they proved their method always produces affine space in dimension $1$ and $2$ \cite[Claim 5.7 and Claim 5.8]{ad}. Infact, for $n \geq 4$, no example of smooth affine $\mathbb{A}^1$-contractible complex variety non isomorphic to $\mathbb{A}^n_{\mathbb{C}}$ is known so far.
%Therefore there are $\mathbb{A}^1$-contractible smooth schemes of dimension at least $6$ which are not affine spaces. 

In dimension three, Koras-Russell threefolds of first kind (\cite{kalimakar}, \cite{korasrussell}) are $\mathbb{A}^1$-contractible \cite[Theorem 1.1]{df} \cite[Theorem 4.2]{hko} but not isomorphic to $\mathbb{A}^3_{\mathbb{C}}$ \cite[Theorem 9.9]{fr}. There are topologically contractible surfaces not isomorphic to $\mathbb{A}^2_{\mathbb{C}}$ \cite[\S 3]{ra} \cite[Theorem A]{dp}. It is therefore natural to ask whether $\mathbb{A}^2$ is the only $\mathbb{A}^1$-contractible variety of dimension $2$ (\cite[Conjecture 5.2.3]{as}). 
         In this article we will prove the following result:
\begin{theorem} \label{main theorem}
A smooth affine surface $X$ over a field $k$ of characteristic zero is isomorphic to $\mathbb{A}^2_{k}$ if and only if $X$ is $\mathbb{A}^1$-contractible.  
\end{theorem}
The idea of the proof is to use the following algebraic characterisation of the affine plane due to Miyanishi-Sugie (\cite[Section 4.1]{miyanishi}).
\begin{theorem}
A non-singular affine surface $X = Spec(A)$ over an algebraically closed field $k$ of characteristic zero is isomorphic to $\mathbb{A}^2_k$ if and only if $A$ is a unique factorisation domain  with $A^*= k^*$ and the (log) Kodaira dimension $\bar{\kappa}(X)$ of $X$ is negative.

\end{theorem}

 The main ingredient comes from Theorem \ref{main}, Section 4, which implies that  if $X$ is $\mathbb{A}^1$-connected affine complex surface then it has negative logarithmic Kodaira dimension. This is the main result in this paper. 
The above characterisation also establishes that $\mathbb{A}^1$-contractibility is indeed a stronger notion than topological contractibility as Ramanujam surface \cite[\S 3]{ra} is not $\mathbb{A}^1$-contractible but topologically contractible, so are the topologically contractible tom Dieck-Petrie surfaces. In particular, we get the the following corollary which answers \cite[Question 6.4]{ad}.

\begin{corollary}
There exists topologically contractible smooth algebraic surfaces which are not $\mathbb{A}^1$-contractible. 

\end{corollary}

%Theorem \ref{main theorem} gives immediate another proofs of the following two classical results:
%\begin{corollary}
%Koras-Russell threefolds of first kind is not a product of two proper subvarieties.
%\end{corollary}
  % Also we get a generalised Zariski's cancellation about affine surface:
%\begin{corollary} \label{Zariski}
%Let $X$ be a smooth complex surface. Suppose $X \times Y$ is $\mathbb{A}^1$-contractible. Then $X \cong \mathbb{A}^2_\mathbb{C}$. In particular, if $X \times Y \cong \mathbb{A}^N_{\mathbb{C}}$, then $X \cong \mathbb{A}^2_{\mathbb{C}}$. 
%\end{corollary}
 
%It is shown that $M(X) \cong M(Spec \ \mathbb{C})$ in $DM_{gm}(\mathbb{C}, \mathbb{Z})$, where $X$ is the Ramanujam surface \cite[Theorem 1]{asok}. This shows that motivic contractibility is weaker than $\mathbb{A}^1$-contractibility. %and it is topologically contractible. Therefore $\mathcal{O}(X)$  is a U.F.D. and it has no non-trivial units, so $X$ cannot be $\mathbb{A}^1$-contractible. Thus there exists topologically contractible but non $\mathbb{A}^1$-contractible surface.% In fact there exists motivically contractible but non $\mathbb{A}^1$-contractible surface.  
\par
\subsection{Organization of the article}  
To prove our main result, we had to analyse $\pi_0^{\mathbb{A}^1}$ of a variety. Section 2, Section 3 and Section 6 grew out of this analysis. 
      In Section 2, we first recall the definition of $\mathbb{A}^1$-connected component sheaf and its related $\mathbb{A}^1$-invariant sheaf $\mathcal{L}(X)$ \cite[Definition 2.9]{bhs}. The main result here is Corollary \ref{comparison}, which says that for a space $\mathcal{X}$, the sheaves $\pi_0^{\mathbb{A}^1}(\mathcal{X})$ and $\mathcal{L}(\pi_0^{\mathbb{A}^1}(\mathcal{X}))$ have same sections over any finitely generated separable field extension $F/k$ (see also  \cite[Theorem 2.2]{cba}). \par
    Asok and Morel constructed the birational $\mathbb{A}^1$-invariant sheaf, denoted by $\pi_0^{b\mathbb{A}^1}(X)$, for a smooth proper scheme $X$ \cite[Section 6]{am}. The sections over any finitely generated separable field extension $L/k$ of $\pi_0^{b\mathbb{A}^1}(X)$ are the same as the $\mathbb{A}^1$-equivalence classes of $L$-points of $X$ (\cite[Theorem 6.2.1]{am}). In Section 3, we prove that the sheaf $\pi_0^{b\mathbb{A}^1}(X)$ is the connected component sheaf of $X$ in the birational model structure on the category of spaces (see Theorem \ref{birational}) which answers the question raised in \cite[Theorem 4]{ks}.   
In Section 5, we prove Theorem \ref{main theorem} together with the characterisations of $\mathbb{A}^3_{\mathbb{C}}$ and $\mathbb{A}^4_{\mathbb{C}}$.
    In Section 6, we introduce a new invariant $\mathcal{O}_{ch}(X)$ for an affine variety $X$ which is a subring of the Makar-Limanov invariant of $X$. Unlike Makar-Limanov invariant, $\mathcal{O}_{ch}(-)$ is functorial and homotopy invariant but it is not representable in $\mathbf{H}(k)$ (see Lemma \ref{not representable}). We prove that it is the ring of regular functions on the $\mathbb{A}^1$-chain connected component sheaf of $X$ (see Proposition \ref{reg}) i.e.
  $$\mathcal{O}_{ch}(X) \cong Hom_{Shv(Sm/k)}(\mathcal{S}(X), \mathbb{A}^1). $$ 
Theorem \ref{ochtrivial} provides evidence that the ring $\mathcal{O}_{ch}(X)$ detects $\mathbb{A}^1$-s in an affine variety $X$. From Section 4 onwards, we assume $k$ to be an algebraically closed field unless otherwise mentioned. \\

\noindent\textbf{Acknowledgements:} The authors would like to thank Aravind Asok, Adrien Dubouloz, R. V. Gurjar, A. J. Parameswaran, Chetan Balwe, Amit Hogadi, Amartya Kumar Dutta, Neena Gupta, V. Srinivas and Mikhail Zaidenberg for their helpful comments and suggestions. We thank the anonymous referee for insightful suggestions which helped us to improve the exposition. This is a part of the Ph.D thesis of the second author and the second author would like to thank Indian Statistical Institute for providing all the resources during this work.

\section{First properties of $\pi_0^{\mathbb{A}^1}$}

%\bigskip

%\begin{remark}
%Denote $(\bigtriangleup^{op} PSh(Sm/k))_{\bullet}$ the category of pointed spaces, i.e. the presheaves of pointed simplicial sets on $Sm/k$. We also have the pointed versions of the local and motivic model structures where weak equivalences are detected after forgetting the base point. The pointed Nisnevich homotopy category is denoted by $\mathbf{H}_{s, \bullet}(k)$. The pointed Nisnevich motivic homotopy category or the pointed motivic homotopy category is denoted by  $\mathbf{H}_{\bullet}(k)$. 
%\end{remark}

%\bigskip
The goal of this section is to show that the field valued points of the $\mathbb{A}^1$-connected component of a space can be computed using an $\mathbb{A}^1$- invariant sheaf associated with the space (see Definition \ref{def: hom inv sheaf}, Definition \ref{def : connected component invariant}, Corollary \ref{comparison}). 
 For any space $\mathcal{X}$, 
define $\pi_0^{\mathbb{A}^1} (\mathcal{X})$ to be the Nisnevich sheaf associated to the presheaf 
$$U \in Sm/k \mapsto Hom_{\mathbf{H}(k)}(U, \mathcal{X}).$$ This is called the $\mathbb{A}^1$-connected component sheaf of $\mathcal{X}$. A space $\mathcal{X}$ is called $\mathbb{A}^1$-connected if $\pi_0^{\mathbb{A}^1}(\mathcal{X}) \cong Spec \ k$ as sheaves.

%Similarly for a pointed space $(\mathcal{X}, x)$ the $i$-the $\mathbb{A}^1$-homotopy sheaf of groups $\pi_{i}^{\mathbb{A}^{1}}(\mathcal{X},x)$ is defined as the Nisnevich sheaf associated to the presheaf $$U \in Sm/k \mapsto Hom_{\mathbf{H}_{\bullet}(k)}(S^i_s \wedge U_+, \mathcal{X});$$
%where $S^i_s$ is the $i$-the simplicial sphere and $U_+$ is the pointed space $U \sqcup \left\{\bullet \right\}$.
       % For a space $\mathcal{X}$ denote by $P^{(n)}(\mathcal{X})$ the simplicial sheaf associated with the presheaf $U \mapsto (\mathcal{X}(U))^{(n)}$ where $K \mapsto K^{(n)} = Im(K \to cosk_n(K))$ is the coskeketon functor \cite[Page $57$]{mv}.
\begin{definition} (\cite[Example 2.4]{mv}, \cite[Definition 7]{mor}) \label{def: hom inv sheaf}
A presheaf of sets $\mathcal{F}$ on $Sm/k$ (resp. a $k$-scheme $X$) is said to be $\mathbb{A}^1$-invariant (resp. $\mathbb{A}^1$-rigid) if for each smooth $k$-scheme $U$, the natural map $\mathcal{F}(U) \to \mathcal{F}(\mathbb{A}^1_U)$ (resp. $X(U) \rightarrow X(\mathbb{A}^1_U))$ induced by the projection map $\mathbb{A}^{1}_U \rightarrow U$ is a bijection.
\end{definition}

\begin{remark} \label{Open-Closed}
\begin{enumerate}
\item The $\mathbb{A}^1$-rigid $k$-schemes are examples of $\mathbb{A}^1$-fibrant spaces. These $\mathbb{A}^1$-fibrant objects have trivial $\mathbb{A}^1$-homotopy sheaf of groups \cite[Example 2.4]{mv}.
\item Two $\mathbb{A}^1$-rigid $k$-schemes are isomorphic in $\mathbf{H}(k)$ if and only if they are isomorphic as $k$-schemes \cite[Lemma 2.1.9]{am}.
\item Abelian varieties, $\mathbb{G}_m$, any smooth projective curve of positive genus are the examples of $\mathbb{A}^1$-rigid $k$-schemes \cite[Example 2.1.10]{am}.
\item  For an $\mathbb{A}^1$-rigid scheme $X$, $\pi_0^{\mathbb{A}^1}(X)$ is isomorphic to $X$ \cite[Lemma 2.1.9]{am}. 
\item Any open or closed subscheme of an $\mathbb{A}^{1}$-rigid $k$-scheme is $\mathbb{A}^{1}$-rigid.
\item Finite product of $\mathbb{A}^{1}$-rigid $k$-schemes is $\mathbb{A}^{1}$-rigid.
\end{enumerate}
\end{remark}

%The next lemma gives a method of constructing $\mathbb{A}^1$-rigid schemes from the known ones.

%\begin{lemma} \label{Open-Closed}
%\begin{enumerate}
%\item Any open or closed subscheme of an $\mathbb{A}^{1}$-rigid $k$-scheme is $\mathbb{A}^{1}$-rigid.
%\item Finite product of $\mathbb{A}^{1}$-rigid $k$-schemes is $\mathbb{A}^{1}$-rigid.
%\end{enumerate}
%\end{lemma}
%\begin{proof}
%\begin{enumerate}
%\item Let $j : U \to X$ (resp. $i : Z \to X$) be open embedding (resp. closed embedding). The claim follows from the fact that the morphism $j$ (resp. $i$) is injective and the morphism $U(V \times_{k} \mathbb{A}^{1}) \to U(V)$ (resp. $Z(V \times_{k} \mathbb{A}^{1}) \to Z(V)$) is surjective for any $V \in Sm/k$ since $X$ is $\mathbb{A}^{1}$-rigid.
%$U(\mathbb{A}^1_X) \to U(X)$ ( resp. $(Z(\mathbb{A}^1_X) \to Z(X)$) is already surjective.

%\item The claim follows from the bijection $(X \times_k Y)(V) \cong X(V) \times Y(V)$ for any $V \in Sm/k$.
%\end{enumerate}
%\end{proof}
The next lemma shows that $\mathbb{A}^1$-homotopy theory of smooth schemes has as building blocks $\mathbb{A}^1$-rigid smooth $k$-schemes. These building blocks have no higher homotopies by Remark \ref{Open-Closed}. This is different from the local nature of \'etale homotopy theory %(Your job to find out the reference for \'etale homotopy type)
and also different from the usual homotopy theory of manifolds.
\begin{lemma} [Local nature] \label{Local Nature}
$X$ is a locally finite type $k$-scheme, then $X$ has a local base of $\mathbb{A}^{1}$-rigid $k$-schemes at each of its points.
\end{lemma}
\begin{proof}
Since $X$ is of locally finite type, $X$ has an open covering by closed subschemes of $\mathbb{A}_k^n$. So it is enough to prove the theorem for $\mathbb{A}_k^n$ by Remark \ref{Open-Closed}. For any point $P \in \mathbb{A}_k^n$, $P$ is in some basic open set $D((x_{1} - \alpha_{1})(x_{2} - \alpha_{2})..(x_{n} - \alpha_{n}))$. This basic open set is a finite product of $\mathbb{G}_m$-s, so it is $\mathbb{A}^1$-rigid by Remark \ref{Open-Closed}. Thus all open subsets of this basic open set form a local base at $P$ of $\mathbb{A}^1$-rigid $k$-schemes by Remark \ref{Open-Closed}.
\end{proof}
Inspired by the discreteness of topological connected components, Morel conjectured the following about $\pi_0^{\mathbb{A}^1}(\mathcal{X})$ \cite[Conjecture 1.12]{mor}.
\begin{conjecture}[Morel]
$\pi_0^{\mathbb{A}^1}(\mathcal{X})$ is $\mathbb{A}^1$-invariant for any space $\mathcal{X}$.
\end{conjecture}
\begin{remark} \label{properties}
$\pi_0^{\mathbb{A}^1}(\mathcal{X})$ is $\mathbb{A}^1$-invariant for the following $\mathcal{X}$:
\begin{enumerate}
\item $\mathcal{X}$ is an $\mathbb{A}^1$-connected space. 
\item $\mathcal{X}$ is a motivic $H$-group or  
a homogeneous space for motivic $H$-group \cite[Theorem 4.18]{c} over an infinite perfect field.  
\item $\mathcal{X}$ is a smooth projective surface \cite[Corollary 3.15, over any field in case of non-uniruled surface] {bhs} \cite[Theorem 1.2, over an algebraically closed field of characteristic zero  in case of birationally ruled surface]{ca} or a smooth toric variety \cite[Lemma 4.2, Lemma 4.4]{w}.
\end{enumerate}
Other than these cases, the conjecture remains open.
\end{remark}
\begin{definition} \cite[Definition 2.9]{bhs} \label{def : connected component invariant}
%Your exercise .
Let $\mathcal{F}$ be a presheaf of sets on $Sm/k$, $\mathcal{S}(\mathcal{F})$ is defined as the Nisnevich sheaf associated to the presheaf $\mathcal{S}^{pre}(\mathcal{F})$ given by
$$\mathcal{S}^{pre}(\mathcal{F})(U) := \mathcal{F}(U)/\sim$$ 
for $U \in Sm/k$, where $\mathcal{F}(U)/\sim$ is the quotient of $\mathcal{F}(U)$ by the equivalence relation generated by $\sigma_0(z) \sim \sigma_1(z)$, $\forall z \in \mathcal{F}(\mathbb{A}^1_U)$ and $\sigma_0, \sigma_1: \mathcal{F}(\mathbb{A}^1_U) \to \mathcal{F}(U)$ are induced by the $0$-section and the $1$-section $U \to \mathbb{A}^1_U$ respectively. For any $n > 1$, $\mathcal{S}^n(\mathcal{F})$ is defined inductively as the sheaves
$$\mathcal{S}^n(\mathcal{F}) := \mathcal{S}(\mathcal{S}^{n-1}(\mathcal{F})).$$
For any sheaf $\mathcal{F}$, there is a canonical epimorphism $\mathcal{F} \to \mathcal{S}(\mathcal{F})$. The sheaf $\mathcal{L}(\mathcal{F})$ is defined as 
$$\mathcal{L}(\mathcal{F}) :=  \underset{n}{\varinjlim} \; \mathcal{S}^n(\mathcal{F}).$$ Therefore there is an induced epimorphism $\mathcal{F} \to \mathcal{L}(\mathcal{F})$.  
%Given a Nisnevich sheaf of sets $\mathcal{F}$ on $Sm/k$, $\mathcal{S}(\mathcal{F})$, the sheaf of $\mathbb{A}^1$-chain connected components of $\mathcal{F}$ is a sheaf of sets associated to the presheaf $U \mapsto $ coequaliser($\mathcal{F}(U \times \mathbb{A}^1) \rightrightarrows \mathcal{F}(U)$) \cite{bhs}[Definition $2.9$], where two maps are induced by $0$-section and $1$-section. $\mathcal{S}^{n}(\mathcal{F})$ is defined inductively as $\mathcal{S}(\mathcal{S}^{n-1}(\mathcal{F}))$ and the sheaf $\mathcal{S}^{\infty}(\mathcal{F}) :=$  $\underset{n}{\varinjlim}$ $\mathcal{S}^n(\mathcal{F})$.
\end{definition}

\begin{remark}
\begin{enumerate} \label{propertiesS(X)}
\item For $X \in Sm/k$, $\mathcal{S}(X)$ is the $\mathbb{A}^1$-chain connected component sheaf $\pi_0^{ch}(X)$ of $X$ which is the Nisnevich sheaf on $Sm/k$ associated to the presheaf $\pi_0(Sing_*^{\mathbb{A}^1}(X))$ \cite[Remark 2.10]{bhs}. So $\mathcal{S}(X)$ is the coequalizer in $Shv(Sm/k)$ of 
$$\begin{tikzcd}
\underline{Hom}(\mathbb{A}^1_k, X) \ar[r, shift left=.75ex, "\theta_0"] \ar[r, shift right=.75ex,swap, "\theta_1"] 
& X
\end{tikzcd}$$
where $\theta_0$ and $\theta_1$ are induced by the $0$-section and the $1$-section $Spec \ k \to \mathbb{A}^1_k$ respectively.
\item  There is a natural map $$\mathcal{S}(X) \to \pi_0^{\mathbb{A}^1}(X)$$ which is an epimorphism.
\item For a presheaf $\mathcal{F}$ on $Sm/k$, $\mathcal{L}(\mathcal{F})$ is an $\mathbb{A}^1$-invariant sheaf \cite[Theorem 2.13]{bhs}.
\item The canonical epimorphism $\mathcal{F} \to \mathcal{L}(\mathcal{F})$ uniquely factors through $\mathcal{F} \to \pi_0^{\mathbb{A}^1}(\mathcal{F})$ \cite[Remark 2.15]{bhs}. The morphism $\pi_0^{\mathbb{A}^1}(\mathcal{F}) \to \mathcal{L}(\mathcal{F})$ is an isomorphism if and only if  $\pi_0^{\mathbb{A}^1}(\mathcal{F})$ is $\mathbb{A}^1$-invariant \cite[Corollary 2.18]{bhs}.
\end{enumerate}
\end{remark}

\begin{definition}
Suppose, $\mathcal{G} \in PSh(Sm/k)$. $\mathcal{G}$ is called homotopy invariant in one variable if for each finitely generated separable field extension $F$ of $k$ the map $\mathcal{G}(F) \to \mathcal{G}(\mathbb{A}^1_F)$ induced by projection is a bijection.
\end{definition}
\begin{example}
\begin{enumerate}
\item  Any $\mathbb{A}^1$-invariant presheaf is homotopy invariant in one variable.
\item $\pi_0^{\mathbb{A}^1}(\mathcal{X})$ is homotopy invariant in one variable for any space $\mathcal{X}$ \cite[Corollary 3.2]{c}.
\end{enumerate}
\end{example}
The following results give a method of comparing $\pi_0^{\mathbb{A}^1}(\mathcal{F})$ and $\mathcal{L}(\mathcal{F})$. Corollary \ref{comparison} was already proved in \cite[Theorem 2.2]{cba}. However our proof works in a more general setting.
\begin{lemma} \label{separable extension}
Suppose $\mathcal{G}$ is a Nisnevich sheaf of sets on $Sm/k$ which is homotopy invariant in one variable and $F/k$ is a finitely generated separable field extension. Then the map $\mathcal{G}(F) \to \mathcal{S}(\mathcal{G})(F)$ is a bijection.
\end{lemma}
\begin{proof}
Surjectivity follows because of the epimorphism $\mathcal{G} \to \mathcal{S}(\mathcal{G})$. For injectivity, suppose $a, b \in \mathcal{G}(F)$ such that $a$ and $b$ map to the same element of $\mathcal{S}(\mathcal{G})(F)$. Then there are chain of $\mathbb{A}^1_F$-s in $\mathcal{G}$ joining $a$ and $b$. 
But any $H \in \mathcal{G}(\mathbb{A}^1_F)$ factors through $\mathcal{G}(F)$. Therefore $a = b$ in $\mathcal{G}(F)$.
\end{proof}

%A Noetherian $k$-scheme $X$ is called an essentially smooth $k$-scheme if it is the inverse limit of a left filtering system $(X_{\alpha})_{\alpha}$ with each transition morphism $X_{\beta} \to X_{\alpha}$ is an \'etale affine morphism between smooth $k$-schemes \cite[Lemma 3.3]{c}. The local schemes $Spec \ \mathcal{O}_{X,x}$ and $Spec \ \mathcal{O}_{X, x}^h$ are essentially smooth $k$-schemes for any $X \in Sm/k$ and for any $x \in X$.
%The proof of the theorem will follow from the following lemma which is the particular case when $X$ is the spectrum of an essentially smooth discrete valuation ring.
\begin{lemma} \label{DVR}
Suppose $\mathcal{G}$ is a Nisnevich sheaf of sets on $Sm/k$ which is homotopy invariant in one variable and $X = Spec \ R$, spectrum of an essentially smooth discrete valuation ring. Then the map $\mathcal{G}(X) \to \mathcal{S}(\mathcal{G})(X)$ is surjective. 
\end{lemma}
\begin{proof}
%First note that, for every finitely generated separable field extension $F$ of $k$, the map $\mathcal{G}(F) \to \mathcal{S}(\mathcal{G})(F)$ is a bijection. Surjectivity follows because of the epimorphism $\mathcal{G} \to \mathcal{S}(\mathcal{G})$. For injectivity, suppose $a, b \in \mathcal{G}(F)$ such that $a$ and $b$ map to the same element of $\mathcal{S}(\mathcal{G})(F)$. Then there are chain of $\mathbb{A}^1_F$-s in $\mathcal{G}$ joining $a$ and $b$. 
%But any $H \in \mathcal{G}(\mathbb{A}^1_F)$ factors through $\mathcal{G}(F)$. Therefore $a = b$ in $\mathcal{G}(F)$. \par
% $ $\exists \ \sigma_1, \sigma_2,.., \sigma_n \in \mathcal{X}(\mathbb{A}^1_F)$ such that $\sigma_{i-1}(1) = \sigma_i(0)$ $\forall i$ and $\sigma_1(0)=\alpha$, $\sigma_n(1)=\beta$. Here $\sigma_i(0)$ and $\sigma_i(1)$ are the images of $\sigma_i$ under the map $\mathcal{X}(\mathbb{A}^1_F) \to \mathcal{X}(F)$ induced by $0$-section and $1$-section $Spec(F) \to \mathbb{A}^1_F$ respectively. Now since the projection map induces bijection $\mathcal{X}(F) \to \mathcal{X}(\mathbb{A}^1_F)$ so, $\sigma_i$-s have lifts $\tilde{\sigma_i} \in \mathcal{X}(F)$. Hence $\sigma_i(0)= \sigma_i(1)=\tilde{\sigma_i}$ $\ \forall i$ which implies $\alpha = \beta$. \par
%Now suppose, $X =Spec \ R$ where $R$ is an essentially smooth discrete valuation ring.
Let $\alpha$ be an element of $\mathcal{S}(\mathcal{G})(X)$. The element $\alpha$ gives an element of $\mathcal{S}(\mathcal{G})(Spec \ R^h)$ ($R^h$ is the Henselization of $R$). The map $\mathcal{G}(Spec \ R^h) \to \mathcal{S}(\mathcal{G})(Spec \ R^h)$ is surjective. So there is a Nisnevich neighbourhood $W \to X$ of the closed point of $X$ and $\alpha^{\prime} \in \mathcal{G}(W)$ such that $\alpha^{\prime}$ maps to $\alpha|_W$. Suppose, $F = Frac(R)$ and $L = K(W)$. Since over $F$ we have bijection $\mathcal{G}(F) \to \mathcal{S}(\mathcal{G})(F)$, there is $\beta \in \mathcal{G}(F)$ such that $\beta$ maps to $\alpha|_F$. The following square is an elementary distinguished square in the Nisnevich topology (Definition \ref{cd}):
$$\begin{tikzcd}
& Spec \ L \arrow[r] \arrow[d]
& W \arrow[d] \\
& Spec \ F \arrow[r]
& X
\end{tikzcd}$$
%We have, $(\alpha|_F)|_L=(\alpha|_W)|_L$. 
Since the morphism $\mathcal{G} \to \mathcal{S}(\mathcal{G})$ is bijection for sections over fields, we have $\beta|_L = \alpha^{\prime}|_L$. As $\mathcal{G}$ is a sheaf, $\beta$ and $\alpha^{\prime}$ lift to an element $\tilde{\alpha} \in \mathcal{G}(X)$. This $\tilde{\alpha}$ maps to $\alpha$.
\end{proof}
\begin{theorem}\label{sur}
Let $\mathcal{G}$ be a Nisnevich sheaf of sets on $Sm/k$ which is homotopy invariant in one variable. Then for each $X \in Sm/k$ with $dim(X) \leq 1$, the map $\mathcal{G}(X) \to \mathcal{S}(\mathcal{G})(X)$ is surjective.
\end{theorem}
\begin{proof}
The proof of the theorem follows from Lemma \ref{DVR} and the Zariski descent argument in the proof in \cite[Theorem 3.1]{c}.
\end{proof}
\begin{corollary} \label{S}
Suppose $\mathcal{G}$ is a Nisnevich sheaf of sets on $Sm/k$ which is homotopy invariant in one variable. Then $\mathcal{S}(\mathcal{G})$ is also homotopy invariant in one variable. 
\end{corollary}
\begin{proof}
Using Lemma \ref{separable extension} and Theorem \ref{sur}, we get the proof.
%We have the following commutative square:
%$$\begin{tikzcd}
%& \mathcal{G}(F) \arrow[r] \arrow[d]
%& \mathcal{S}(\mathcal{G})(F) \arrow[d] \\
%& \mathcal{G}(\mathbb{A}^1_F) \arrow[r]
%& \mathcal{S}(\mathcal{G})(\mathbb{A}^1_F)
%\end{tikzcd}$$
%where the vertical maps are induced by the projection map $\mathbb{A}^1_F \to F$. Here the top most map is a bijection follows from the proof of Lemma \ref{DVR} and the left most map is a bijection since $\mathcal{G}$ is homotopy invariant in one variable. The right vertical map is an injection since it is induced by projection map and the remaining horizontal map is a surjection by Theorem \ref{sur}. So the map $\mathcal{S}(\mathcal{G})(F) \to \mathcal{S}(\mathcal{G})(\mathbb{A}^1_F)$ is surjective and hence it is a bijection.
\end{proof}
\begin{corollary} \label{inf}
Suppose $\mathcal{G}$ is a Nisnevich sheaf of sets on $Sm/k$ which is homotopy invariant in one variable and $F/k$ is a finitely generated separable field extension. Then the maps $\mathcal{G}(F) \to \mathcal{L}(\mathcal{G})(F)$ and $\mathcal{G}(\mathbb{A}^1_F) \to \mathcal{L}(\mathcal{G})(\mathbb{A}^1_F)$ are bijections.
\end{corollary}
\begin{proof}
The map $\mathcal{G}(F) \to \mathcal{S}(\mathcal{G})(F)$ is a bijection by Lemma \ref{separable extension}. Since $\mathcal{S}(\mathcal{G})$  is homotopy invariant in one variable by Corollary \ref{S}, the map $\mathcal{S}(\mathcal{G})(F) \to \mathcal{S}^2(\mathcal{G})(F)$ is a bijection. By induction we get, the map $$\mathcal{S}^n(\mathcal{G})(F) \to \mathcal{S}^{n+1}(\mathcal{G})(F)$$ is a bijection $\forall \ n$. Since $\mathcal{L}(\mathcal{G})$ is the colimit of $\mathcal{S}^n(\mathcal{G})$, the map $\mathcal{G}(F) \to \mathcal{L}(\mathcal{G})(F)$ is a bijection. As $\mathcal{L}(\mathcal{G})$ is an $\mathbb{A}^1$-invariant sheaf (\cite[ Theorem 2.13]{bhs}) and $\mathcal{G}$ satisfies $\mathcal{G}(F) \cong \mathcal{G}(\mathbb{A}^1_F)$, the map $\mathcal{G}(\mathbb{A}^1_F) \to \mathcal{L}(\mathcal{G})(\mathbb{A}^1_F)$ is a bijection.
\end{proof}
\begin{corollary} \label{comparison}
Suppose $\mathcal{X} \in \Delta^{op}PSh(Sm/k)$. The maps $\pi_0^{\mathbb{A}^1}(\mathcal{X})(F) \to \mathcal{L}(\pi_0^{\mathbb{A}^1}(\mathcal{X}))(F)$ and $\pi_0^{\mathbb{A}^1}(\mathcal{X})(\mathbb{A}^1_F) \to \mathcal{L}(\pi_0^{\mathbb{A}^1}(\mathcal{X}))(\mathbb{A}^1_F)$ are bijections for any finitely generated separable field extension $F/k$.
\end{corollary}
\begin{proof}
The canonical morphism 
$$\Tilde{\pi}_0^{\mathbb{A}^1}(\sX)(\mathbb{A}^1_F) \to  \pi_0^{\mathbb{A}^1}(\sX)(\mathbb{A}^1_F)$$
is a bijection for any finitely generated separable field extension $F/k$ by \cite[Corollary 3.2]{c} (here $\Tilde{\pi}_0^{\mathbb{A}^1}(\sX)$ is the presheaf on $Sm/k$ defined as $U \in Sm/k \mapsto Hom_{\mathbf{H}(k)}(U, \mathcal{X})$). By definition, the presheaf $\Tilde{\pi}_0^{\mathbb{A}^1}(\sX)$ is homotopy invariant in one variable. Therefore, the sheaf $\pi_0^{\mathbb{A}^1}(\mathcal{X})$ is homotopy invariant in one variable.
%that $\pi_0^{\mathbb{A}^1}(\sX)$ is homotopy invariant in one variable by considering the following diagram:
%$$\begin{tikzcd}
%& \Tilde{\pi}_0^{\mathbb{A}^1}(\sX)(F) \arrow[r] \arrow[d]
%& \Tilde{\pi}_0^{\mathbb{A}^1}(\sX)(\mathbb{A}^1_F) \arrow[d] \\
%& \pi_0^{\mathbb{A}^1}(\sX)(F) \arrow[r]
%& \pi_0^{\mathbb{A}^1}(\sX)(\mathbb{A}^1_F)
%\end{tikzcd}$$
%and thus $\pi_0^{\mathbb{A}^1}(\sX)$ is the Nisnevich sheafification of $\Tilde{\pi_0}^{\mathbb{A}^1}(\sX)$. 
%The upper and lower horizontal maps are induced by the projection $\mathbb{A}^1_F \to F$. The upper horizontal arrow is bijection from the definition of $\Tilde{\pi_0}^{\mathbb{A}^1}(\sX)$. The left vertical arrow is bijection since the section over $F$ are stalks. The right vertical arrow is surjection by \cite[Theorem 3.1]{c}. Thus the lower horizontal arrow which is already injective since there is a $0$-section $F \to \mathbb{A}^1_F$, is surjective. 
Hence the corollary follows from Corollary \ref{inf}. 
 %We only need to show that the sheaf $\pi_0^{\mathbb{A}^1}(\mathcal{Z})$ satisfies the hypothesis. This is because of the following commutative square:
%$$\begin{tikzcd}
%& \tilde{\pi_0}^{\mathbb{A}^1}(\mathcal{Z})(F) \arrow[r] \arrow[d]
%& \tilde{\pi_0}^{\mathbb{A}^1}(\mathcal{Z})(\mathbb{A}^1_F) \arrow[d] \\
%& \pi_0^{\mathbb{A}^1}(\mathcal{Z})(F) \arrow[r]
%& \pi_0^{\mathbb{A}^1}(\mathcal{Z})(\mathbb{A}^1_F)
%\end{tikzcd}$$
%where $\tilde{\pi_0}^{\mathbb{A}^1}(\mathcal{Z})$ is the presheaf of $\pi_0^{\mathbb{A}^1}(\mathcal{Z})$ defined as $U \mapsto Hom_{\mathbf{H}(k)}(U, \mathcal{Z})$ and the horizontal maps are induced by the projection map $\mathbb{A}^1_F \to F$. The upper horizontal map is a bijection follows from the definition of the presheaf $\tilde{\pi_0}^{\mathbb{A}^1}(\mathcal{Z})$. The right vertical map is a bijection [\cite{c}, Corollary $3.2$] and the left vertical map is a bijection since the sections over finitely generated separable field extensions are stalks. Therefore the map $\pi_0^{\mathbb{A}^1}(\mathcal{Z})(F) \to \pi_0^{\mathbb{A}^1}(\mathcal{Z})(\mathbb{A}^1_F)$ is a bijection.
\end{proof}
\begin{question}
For a proper scheme $X \in Sm/k$, we have $\mathcal{S}(X)(F) \cong \mathcal{S}^2(X)(F)$ for each finitely generated separable extension $F$ of $k$ \cite[Theorem 3.9]{bhs}. Thus $\mathcal{S}(X)$ is homotopy invariant in one variable. On the other hand, the real sphere $T$ in $\mathbb{A}^3_{\mathbb{R}}$ contains no non-constant $\mathbb{A}^1_{\mathbb{R}}$ but $\mathcal{S}^2(T)(\mathbb{R})$ is point (\cite[Theorem 4.3.4]{sawant}). Therefore it is natural to ask whether $\mathcal{S}(X)$ is homotopy invariant in one variable for any scheme $X \in Sm/k$ with $k = \bar{k}$.
%equivalently the sections of $\mathcal{S}(X)$ and $\mathcal{S}^2{X}$ over $\mathbb{A}^1_F$ are same. Since 
\end{question}
\section{Birational Connected Components}
The first example of an $\mathbb{A}^1$-invariant sheaf associated to the $\mathbb{A}^1$-connected components of a scheme was constructed by Asok and Morel.
\begin{definition} \cite[Definition 6.1.1]{am}
A presheaf of sets $\mathcal{F}$ on $Sm/k$ is called birational if it satisfies the following properties:
\begin{enumerate}
 \item For $X \in Sm/k$ with irreducible components $X_1, X_2, \dots, X_n$ the canonical map
            $$\mathcal{F}(X) \to \prod_{i=1}^n \mathcal{F}(X_i)$$
is a bijection.
 \item For an open dense subscheme $U$ of $X$, the canonical morphism $\mathcal{F}(X) \to \mathcal{F}(U)$ is a bijection.
\end{enumerate}
\end{definition}
A birational presheaf of sets is always a Nisnevich sheaf \cite[Lemma 6.1.2]{am}.
 Let $X \in Sm/k$ be a proper scheme. There is a birational (thus homotopy invariant \cite[Theorem 6.1.7]{am}) sheaf $\pi_0^{b\mathbb{A}^1}(X)$ \cite[Section 6.2]{am} such that its sections over any finitely generated separable field extension $L$ of $k$ is the $\mathbb{A}^1$-chain connected component of $L$-rational points, i.e. $\mathcal{S}(X)(L) = \pi_0^{b\mathbb{A}^1}(X)(L)$ (\cite[Theorem 6.2.1]{am}).
\begin{remark}
\begin{enumerate}
\item For a proper scheme $X \in Sm/k$, there is a canonical morphism $\pi_0^{\mathbb{A}^1}(X) \to \pi_0^{b\mathbb{A}^1}(X)$ which is a bijection on sections over any finitely generated separable field extensions of $k$ \cite[Proposition 6.2.6]{am}, \cite[Corollary 3.10]{bhs} . 
\item $\pi_0^{b\mathbb{A}^1}$ is a birational invariant of smooth proper schemes \cite[Theorem 1]{ks}. However $\pi_0^{\mathbb{A}^1}$ is a not birational invariant sheaf of smooth proper schemes \cite[Example 4.8]{bhs}.
\end{enumerate}
\end{remark}
In this section in Theorem \ref{birational}, we will prove that $\pi_0^{b\mathbb{A}^1}(X)$ is isomorphic to the connected component sheaf of $X$ in the birational model structure (Proposition \ref{model}). This gives a proof of \cite[Theorem 4]{ks}. Same line of arguement is used in \cite[Proposition 1.9]{ck}. In \cite[Definition 2.6]{pp} Pablo also constructed birational unstable motivic homotopy category (equivalent construction by Theorem \ref{equivalent})
% however our model structure is equivalent with this model structure.  We will call our model structure as the unstable birational model structure.  
\begin{proposition} \label{model}
The left Bousfield localisation of the projective model structure on $\bigtriangleup^{op}PSh(Sm/k)$ with respect to the following class of maps  
$$\{U \xrightarrow{i} X \in Sm/k \text{ }| \text{ } i \text{ is an open immersion with dense image} \}$$ exists. It gives a model structure on $\bigtriangleup^{op}PSh(Sm/k)$ called the unstable birational model structure.
\end{proposition}
\begin{proof}
Existence of the left Bousfield localisation is proved in \cite[Theorem 4.1.1]{hir}.
\end{proof}
The resulting homotopy category associated to the birational model structure will be denoted by $\mathbf{H}_b(k)$.
\begin{definition}
For any space $\mathcal{X}$, the connected component presheaf associated to the birational model structure is defined as $U \mapsto Hom_{\mathbf{H}_b(k)}(U, \mathcal{X})$, for $U \in Sm/k$. It will be denoted by $\pi_0^b(\mathcal{X})$. 
\end{definition}
 The aim of this section is to prove the following result:
\begin{theorem} \label{birational}
There is an  isomorphism of presheaves: $\pi_0^{b\mathbb{A}^1}(X) \cong \pi_0^b(X)$, for $X \in Sm/k$ a proper scheme.
\end{theorem}
%\begin{proof}
%It is the left Bousfield localisation of the projective model structure with respect to the set of maps $\{U \xrightarrow{i} X \text{ }| \text{ } i \text{ is an open immersion with dense image} \}$.

%\end{proof}
%We will use the following basic facts about left Bousfield localisation in the next subsection whose proofs are immediate from definition.
%\begin{lemma} \label{basic}
%Suppose $\mathcal{M}$ is a model category and $S, T$ are two class of maps in $\mathcal{M}$ such that with respect to $S$ and $T$ the left Bousfield localisation of $\mathcal{M}$ exists. Suppose that every element if $S$ is a $T$-local equivalence. Then we have the following:
%\begin{enumerate}
%\item Every $T$ local object is an $S$-local object in $\mathcal{M}$.
%\item Every $S$-local equivalence is a $T$-local equivalence in $\mathcal{M}$.
%\end{enumerate}
%\end{lemma}
%\subsection{Few important classes of weak equivalences in birational model structure}
Let $f : U \to X \in Sm/k$ be a Nisnevich covering and  $U_{\bullet}$ be the corresponding \v{C}ech simplicial scheme. Here $U_n$ is the smooth scheme given by $U \times_X U \times_X \dots \times_X U$ (the product is taken $n+1$ times). Let $f : U_{\bullet} \to X$ be the corresponding map of simplicial schemes. We show that inverting the birational morphisms in $\mathbb{A}^1$-homotopy category is equivalent to only inverting the birational morphisms in the global projective model structure (Theorem \ref{equivalent}).

\begin{lemma} \label{cover}
The map $f : U_{\bullet} \to X$ is a birational weak equivalence (i.e., it is an isomorphism in $\mathbf{H}_{b}(k)$). 
\end{lemma}
%To prove the above theorem we need the following lemma.
%\begin{lemma} ([\cite{v}, Lemma $5.1$]) \label{sec}
%Let $f:X \to Y$ be a morphism between schemes which has a section i.e. $\exists \ g: Y \to X$ such that $f \circ g$ is the identity. Then the corresponding map $X_{\bullet} \to Y$ between simplicial schemes is a simplicial homotopy equivalence.
%\end{lemma}
\begin{proof}
Any Nisnevich covering has a section over a dense open set. Therefore there is an open dense set $V \subset X$ such that the restriction $f^{-1}(V) \to V$ has a section. We have the following commutative diagram in $\Delta^{op}PSh(Sm/k)$:
$$\begin{tikzcd}
& f^{-1}(V)_{\bullet} \arrow[r] \arrow[d]
& U_{\bullet} \arrow[d, "f"] \\
& V \arrow[r]
& X
\end{tikzcd}$$
where the left vertical map is induced by the restriction and the upper horizontal map is induced by the inclusion. The left vertical map is a sectionwise weak equivalence, since there is a section. The map $V \to X$ is an inclusion of dense open set, so it is a birational weak equivalence. As the map $f: U \to X$ is an \'{e}tale map, for each $n$, the $(n+1)$-fold product $f^{-1}(V) \times_V f^{-1}(V)...\times_V f^{-1}(V)$ is open and dense in $U \times_X U..\times_X U$ fitting in the pullback square:
$$\begin{tikzcd}
& f^{-1}(V) \times_V..\times_Vf^{-1}(V) \arrow[r] \arrow[d]
& U \times_X U..\times_X U \arrow[d] \\
& V \arrow[r]
& X
\end{tikzcd}$$
Therefore the morphism $f^{-1}(V)_{\bullet} \to U_{\bullet}$ is a birational weak equivalence \cite[Proposition 2.14]{mv}. Hence $f: U_{\bullet} \to X$ is a birational weak equivalence.  
\end{proof}
\begin{corollary}
Any Nisnevich weak equivalence is a birational weak equivalence.
\end{corollary}
\begin{proof}
The local projective model structure on $\Delta^{op}PSh(Sm/k)$ is the left Bousfield localisation of the projective model structure at the class of \v{C}ech hypercovers (\cite[Theorem 6.2, Example A.11]{dhi}),
\[
\{U_{\bullet} \to X  \mid U \to X \text{ is a Nisnevich covering}\}
\]
Since the map $U_{\bullet} \to X$ is a birational weak equivalence by Lemma \ref{cover}, the result follows.
\end{proof}
\begin{lemma} \label{p}
For every $X \in Sm/k$, the projection map $X \times \mathbb{A}^1 \to X$ is a birational weak equivalence.
\end{lemma}
%To prove this we will use the following theorem from [\cite{ks}, Appendix A].
%\begin{theorem}
%Suppose $\mathcal{F}$ is a presheaf of sets on $Sm/k$ such that for each $g:X \to Y$ is a birational map between projective, geometrically connected surfaces over $k$, the map $\mathcal{F}(Y) \to \mathcal{F}(X)$ is a bijection. Then $\mathcal{F}$ induces a bijection $\mathcal{F}(k) \to \mathcal{F}(\mathbb{P}^1_k)$.
%\end{theorem}
\begin{proof}
Suppose $\mathcal{X} \in \Delta^{op}PSh(Sm/k)$ and $X \in Sm/k$. Consider the presheaf of sets $\mathcal{F}_{\mathcal{X}, X}$ on $Sm/k$ defined as $Y \mapsto Hom_{\mathbf{H}_b(k)}(Y \times X, \mathcal{X})$. Then $\mathcal{F}_{\mathcal{X},X}$ is a birational sheaf on $Sm/k$. Therefore we have a bijection $ \mathcal{F}_{\mathcal{X}, X}(Spec \ k) \to \mathcal{F}_{\mathcal{X}, X}(\mathbb{P}^1_k)$ \cite[Appendix A]{ks}. This implies the projection map $\mathbb{P}^1_k \times X \to X$ is an isomorphism in $\mathbf{H}_b(k)$. Composing it with the birational map $\mathbb{A}^1_k \times X \hookrightarrow \mathbb{P}^1_k \times X$ (induced by the natural open immersion $\mathbb{A}^1_k \hookrightarrow \mathbb{P}^1_k$), we get that the projection map $X \times \mathbb{A}^1_k \to X$ is an isomorphism in $\mathbf{H}_b(k)$.
\end{proof}
\begin{theorem} \label{equivalent}
Any $\mathbb{A}^1$-weak equivalence is a birational weak equivalence. Therefore the unstable birational model structure is equivalent to the motivic unstable birational model structure in \cite[Definition 2.6]{pp}.
\end{theorem}
\begin{proof}
The left Bousfield localisation of the projective model structure (universal model structure) on $\Delta^{op}PSh(Sm/k)$ at the class of the \v{C}ech hypercovers and the projection maps $\mathbb{A}^1_X \to X \in Sm/k$ gives the $\mathbb{A}^1$-model structure \cite[Proposition 8.1]{du}. Therefore, the $\mathbb{A}^1$-weak equivalences are generated by the \v{C}ech hypercovers and the projection maps $\mathbb{A}^1_X \to X$. Both are birational weak equivalences [Lemma \ref{cover} and Lemma \ref{p}]. Hence the result follows.
\end{proof}
\subsection{Proof of the Theorem \ref{birational}}

\begin{proof}
Suppose $U \in Sm/k$ is irreducible. Then,  $\pi_0^{b\mathbb{A}^1}(X)(U)= \mathcal{S}(X)(k(U))$ by \cite[Definition 6.2.5]{am}. Recall the fine birational category of smooth $k$-schemes $S_b^{-1}Sm_{/k}$ \cite[Section 1.7]{ks} which is defined as the localisation of $Sm/k$ with respect to the class of birational morphisms $S_b$. By \cite[Theorem 6.6.3]{ks}, we have the following natural bijection $$Hom_{S_b^{-1}Sm_{/k}}(U, X) \cong \pi_0^{b\mathbb{A}^1}(X)(U),$$ for each $U \in Sm/k$. The Yoneda embedding of $Sm/k$ in $\Delta^{op}PSh(Sm/k)$ as representable constant simplicial presheaf induces a functor $\eta : S_b^{-1}Sm_{/k} \to \mathbf{H}_b(k)$ because of the universal property of localisation \cite[Section 1]{gz}). The functor $\eta$ is universal and it factors the functor $\pi : Sm/k \to \mathbf{H}_b(k)$. This gives the map $$Hom_{S_b^{-1}Sm_{/k}}(U,X) \to Hom_{\mathbf{H}_b(k)}(U, X).$$ Thus we have a morphism $\eta: \pi_0^{b\mathbb{A}^1}(X) \to \pi_0^b(X)$.
 %But $U$ is cofibrant in projective model structure and if $X$ is fibrant in birational model structure then we can identify $Hom_{\mathbf{H}_b(k)}(U, X)$ as $\mathbb{A}^1$- homotopy classes of maps $U \to X$ since $U \times \mathbb{A}^1$ is a cylinder object in birational model structure ($U \sqcup U \hookrightarrow U \times \mathbb{A}^1$ is a projective cofibration and $U \times \mathbb{A}^1 \to U$ is birational weak equivalence). So the map $Hom_{S_b^{-1}Sm}(U, X) \to [U, X]/\sim$ ($\mathbb{A}^1$-homotopy classes of maps) is surjectve. If two maps $k(U) \to X$ are such that their extensions $U \to X$ are $\mathbb{A}^1$-homotopic then they are in same $\mathbb{A}^1$-equivalence class. So the map $X(k(U)) \to [U, X]/\sim$ is injective. Hence we have a bijection $\pi_0^{b\mathbb{A}^1}(X)(U) \to Hom_{\mathbf{H}_b(k)}(U, X)$, $\forall U \in Sm/k$ irreducible.
 \par
Consider the following commutative diagram of presheaves on $Sm/k$,
 $$\begin{tikzcd}
& X \ar[r, "\alpha"] \ar[d, "\pi" left]
& \pi_0^{b\mathbb{A}^1}(X) \ar[d , "\cong"] \\
& \pi_0^b(X) \ar[r] \ar[ur, dotted, "\theta"]
& \pi_0^b(\pi_0^{b\mathbb{A}^1}(X))
\end{tikzcd}$$
 induced by the natural transformation $Id \to \pi_0^b()$.  The top horizontal morphism of presheaves $\alpha : X \to \pi_0^{b\mathbb{A}^1}(X)$ is induced by the canonical functor $\alpha : Sm/k \to S_b^{-1}Sm_{/k}$. The right most vertical map is an isomorphism, since $\pi_0^{b\mathbb{A}^1}(X)$ is a fibrant object in the birational model structure. This gives a morphism $\theta : \pi_0^b(X) \to \pi_0^{b\mathbb{A}^1}(X)$. 
%This diagram is also universal in the following sense: the natural morphism $X \to \pi_0^{b\mathbb{A}^1}(X)$ uniquely factors through $\pi_0^b(X)$. 
The morphism $\eta \circ \theta \circ \pi$ is same as $\eta \circ \alpha$ and $\eta \circ \alpha$ is the natural morphism $\pi$.  The morphism $\pi: X \to \pi_0^b(X)$ induces a bijection $$Hom_{PSh(Sm/k)}(\pi_0^b(X), \pi_0^b(X)) \cong Hom_{PSh(Sm/k)}(X, \pi_0^b(X)),$$ since $\pi_0^b(X)$ is birational local. This gives $\eta \circ \theta$ is the identity morphism. So $\theta$ is a monomorphism. On the other hand, the morphism $\alpha$ factors through $\theta$ and the morphism $\alpha$ is sectionwise surjective, since $\pi_0^{b\mathbb{A}^1}(X)$ is a birational sheaf and its section over $U$ is the $\mathbb{A}^1$-equivalence classes of $k(U)$-rational points of $X$. Hence $\theta$ is an epimorphism and consequently it is an isomorphism. 
\end{proof}
\section{Existence of $\mathbb{A}^1$ and $\mathbb{A}^1$-connectedness}
The aim of this section is to relate $\mathbb{A}^1$-homotopy theory with the existence of affine lines in a variety. The main result in this section is Theorem \ref{main} where we establish this. By the phrase ``there is an $\mathbb{A}^1$ in $X$'', we mean the existence of a non-constant morphism from $\mathbb{A}^1_k$ to $X$. First we recall three important notions of being dominated by the images of affine lines. For this section, we will assume our base field $k$ is an algebraically closed field. 
\begin{definition} \label{several A1}
Suppose $X$ is a $k$-variety, where $k$ is a field of characteristic zero.
\begin{enumerate}
\item $X$ is said to be \textbf{dominated by images of $\mathbb{A}^1$} if there is an open dense subset $U$ of $X$ such that for every $p \in U(k)$, there is an $\mathbb{A}^1$ in $X$ through $p$ \cite[\S 1]{km}. 
 \item $X$ is said to be \textbf{$\mathbb{A}^1$-uniruled or log-uniruled} if there is a dominant generically finite morphism $H: \mathbb{A}^1_k \times_k Y \to X$ for some $k$-variety $Y$. 
\item $X$ is said to be \textbf{$\mathbb{A}^1$-ruled} if there is a Zariski open dense subset $U$ of $X$ such that $U$ is isomorphic to $\mathbb{A}^1_k \times_k Z$ for some $k$-variety $Z$ \cite[Definition 1]{at}.
\end{enumerate}
\end{definition}
%\begin{definition}
%\begin{enumerate}
%\item A $k$-variety $X$ is called log-uniruled or $\mathbb{A}^1$-uniruled if there is an open dense subset $U$ of $X$ such that every $k$-rational point of $U$ lies in the image of some non-constant map from $\mathbb{A}^1_k$ to $X$ \cite[Definition 1]{at}.
%\item A $k$-variety $X$ is called $\mathbb{A}^1$-ruled if there is an open dense subset $V$ of $X$ which is isomorphic to $\mathbb{A}^1 \times U$ for some $k$-variety $U$ \cite[Definition 3]{at}.
%\end{enumerate}
%\end{definition}
%\begin{example}
%Affine spaces $\mathbb{A}^n_k$-s are log-uniruled. $\mathbb{G}_m \times \mathbb{A}^n_k$, for $n \geq 1$, $\mathbb{A}^n_k$ minus finitely many closed points are log-uniruled.
%\end{example}
%\begin{definition}

%\end{definition}
\begin{remark} \label{dominated and uniruled}
Here we will describe few important relations between the above notions. By definition, the $\mathbb{A}^1$-ruled varieties are $\mathbb{A}^1$-uniruled and $\mathbb{A}^1$-uniruled varieties are dominated by images of $\mathbb{A}^1$. Suppose the field $k$ is uncountable, then the smooth varieties dominated by images of $\mathbb{A}^1$ are $\mathbb{A}^1$-uniruled. This is essentially similar to the fact that a smooth projective variety dominated by images of $\mathbb{P}^1$ is uniruled \cite[Chapter IV, Proposition 1.3]{kollar}. Indeed, given a smooth $k$-variety $X$ with fixed smooth completion $\bar{X}$ with boundary $D = \bar{X} \setminus X$, non-constant morphisms from $\mathbb{A}^1_k$ to $X$ are parametrized by a certain subscheme $Mor((\mathbb{P}^1_k, \infty), (\bar{X}, D))$ of the hom scheme $Mor_k(\mathbb{P}^1_k, \bar{X})$ parametrizing morphisms $f: \mathbb{P}^1_k \to \bar{X}$ such that $f^{-1}(D) = \{\infty\}$. There is a canonical evaluation morphism $ev: Mor((\mathbb{P}^1_k, \infty), (\bar{X}, D)) \times_k (\mathbb{P}^1_k \setminus \{\infty\}) \to X$ which is dominant, as $X$ is dominated by images of $\mathbb{A}^1$. Since $k$ is uncountable and $Mor((\mathbb{P}^1_k, \infty), (\bar{X}, D))$ has only countably many irreducible components and there is a dense open subset $U$ of $X$ which is contained in the image of $ev$, we get an irreducible component $Y$ of $Mor((\mathbb{P}^1_k, \infty), (\bar{X}, D))$ such that the restriction to $Y$ of $ev$ 
%the canonical evaluation morphism $Mor((\mathbb{P}^1_k, \infty), (\bar{X}, D)) \times_k (\mathbb{P}^1_k \setminus \{\infty\}) \to X$
is a dominant morphism $\mathbb{A}^1_k \times_k Y \to X$. In case of $k$ is countable field, the equivalence of $(1)$ and $(2)$ is not known. \par
 The varieties which are $\mathbb{A}^1$-uniruled have negative logarithmic Kodaira dimension \cite[Proposition 1]{iit}. If $X$ is a surface and $k$ is uncountable, $(1)$ is equivalent to the negativity of logarithmic Kodaira dimension $\bar{\kappa}(X)$ \cite[Theorem 1.1]{km}. In case of smooth affine $k$-surface, if $\bar{\kappa}(X) = -\infty$, $X$ is $\mathbb{A}^1$-ruled \cite[\S 4, \S 5]{misu}. It is not known in general whether varieties dominanted by images of $\mathbb{A}^1$ have negative logarithmic Kodaira dimension if $k$ is countable. 

\it{Thus for a smooth affine $k$-surface $X$ over an uncountable algebraically closed field $k$ of characteristic zero, we have $(1), (2), (3)$ are equivalent and these are equivalent to the negativity of logarithmic Kodaira dimension}.  However in higher dimensions, being $\mathbb{A}^1$-ruled is a stronger notion than dominated by images of $\mathbb{A}^1$ \cite[Proposition 9]{at}. 
%Here we will see the connection of these notions with $\mathbb{A}^1$-homotopy theory.
\end{remark}
Suppose $\mathcal{F}$ is a Nisnevich sheaf of sets on $Sm/k$ and $W \in Sm/k$, $f \in \mathcal{F}(W)$.
\begin{definition}
An element $\alpha \in \mathcal{F}(Spec \ k)$ is in the image of $f$ if $\exists \ \gamma \in W(Spec \ k)$ such that the composition $Spec \ k \to W \xrightarrow{f} \mathcal{F}$ is $\alpha$.
\end{definition}
\begin{definition}
A homotopy $H \in \mathcal{F}(\mathbb{A}^1_W)$ is said to be non-constant if $H(0) \neq H(1) \in \mathcal{F}(W)$, where $H(0)$ and $H(1)$ are induced by the $0$-section and the $1$-section from $W$ to $\mathbb{A}^1_W$ respectively.
\end{definition}
\begin{remark} \label{useful remark}
 Let $\mathcal{F}$ be a sheaf and $X \in Sm/k$. A section $\alpha \in S(\mathcal{F})(X)$ is given by a  Nisnevich covering $W \to X$, a section $\gamma \in \mathcal{F}(W)$ and a Nisnevich covering $W' \to W \times_{X} W$ such that $ p_1^*(\gamma)|_{W'} $ and $p_2^*(\gamma)|_{W'}$  in $\mathcal{F}(W')$ are joined by a chain of $\mathbb{A}^1$-homotopies (where $p_1, p_2 : W \times_{X} W \to W$ are the projection maps). If  $ p_1^*(\gamma)|_{W'} = p_2^*(\gamma)|_{W'}$, then $\gamma$ can be lifted to some element $\alpha' \in \mathcal{F}(X)$ and in this case $\alpha'$ maps to $\alpha$ via the canonical morphism $\mathcal{F} \to S(\mathcal{F})$.  Otherwise, we will get an element $H \in  \mathcal{F}(\mathbb{A}^1_{W'})$ such that $ p_1^*(\gamma)|_{W'} = H(0) \neq H(1)$ as sections. This is essentially the data of ghost homotopy mentioned in \cite[Definition 3.2]{bhs}. 
\end{remark}
\begin{condition} \label{SIGMA}
Suppose $X, W \in Sm/k$, $\alpha \in X(k)$ and $n \geq 0$. A homotopy $H \in \mathcal{S}^n(X)(\mathbb{A}^1_W)$ is said to satisfy the condition $*(W, \alpha)$, if $H$ satisfies the following properties:
\begin{enumerate}
 \item $H$ is a non-constant homotopy.
 \item $H(0)$ factors through $X$ i.e. there is a morphism $\psi : W \to X$ such that the following diagram commutes:
  $$\begin{tikzcd}
 W \ar[r, "\psi"] \ar[d, "i_0"] &X \ar[d]\\
\mathbb{A}^1_{W} \ar[r,"H"] 
&\mathcal{S}^n(X)
\end{tikzcd}$$ 
Here $i_0: W \to \mathbb{A}^1_W$ is the $0$-section and the right vertical map is the canonical epimorphism $X \to \mathcal{S}^n(X)$.
  \item $\alpha \in \overline{Im(H(0))}$ (By (2), $H(0): W \to X$). 
\end{enumerate}
\end{condition}
\begin{proposition} \label{track}
Suppose $X, W \in Sm/k$, $\alpha \in X(k)$ and $n \geq 1$. Let $H \in \mathcal{S}^n(X)(\mathbb{A}^1_W)$ be a homotopy, where $W$ is irreducible and $H$ satisfies $*(W, \alpha)$. Then there is $W^{\prime} \in Sm/k$ irreducible and a homotopy $H^{\prime} \in \mathcal{S}^m(X)(\mathbb{A}^1_{W^{\prime}})$ for some $m < n$ such that $H^{\prime}$ satisfies $*(W^{\prime},\alpha)$.
 \end{proposition}
\begin{proof}
The morphism $X \to \mathcal{S}^n(X)$ is an epimorphism and $H \in \mathcal{S}^n(X)(\mathbb{A}^1_W)$. Thus, 
\begin{enumerate}
  \item  There is a Nisnevich covering $f:V \to \mathbb{A}^1_W$,
  \item  There is a morphism $\phi: V \to X$ such that the following diagram commutes:
 $$\begin{tikzcd}
 V \ar[r, "\phi"] \ar[d, "f"] &X \ar[d]\\
\mathbb{A}^1_{W} \ar[r,"H"] 
&\mathcal{S}^n(X)
\end{tikzcd}$$ 
\end{enumerate}
The morphism $\phi$ gives an element of $\mathcal{S}^{n-1}(X)(V)$ via the epimorphism $X \to \mathcal{S}^{n-1}(X)$. The elements $p_1^*(\phi)$ and $p_2^*(\phi)$ are same in $\mathcal{S}^n(X)(V \times_{\mathbb{A}^1_W} V)$ (where $p_1, p_2: V \times_{\mathbb{A}^1_W} V \to V$ are the projection maps). Therefore, there is a Nisnevich covering $V^{\prime} \to V\times_{\mathbb{A}^1_W} V$ and there is a chain of non-constant (since $H$ is a non-constant homotopy, so $p_1^{*}(\phi)|_{V^{\prime}} \neq p_2^*(\phi)|_{V^{\prime}} \in \mathcal{S}^{n-1}(X)(V^{\prime})$ by Remark \ref{useful remark}) $\mathbb{A}^1$-homotopies $G_1, G_2, \dots, G_k \in \mathcal{S}^{n-1}(X)(\mathbb{A}^1_{V^{\prime}})$ such that 
$$G_1(0) = p_1^*(\phi)|_{V^{\prime}} \text{ and } G_k(1) = p_2^*(\phi)|_{V^{\prime}}.$$ \par
  Suppose $V = \coprod_{i=1}^n V_i$, $V_i$-s are the irreducible components of $V$. Then $V \times_{\mathbb{A}^1_W} V$ is the union of $V_i \times_{\mathbb{A}^1_W} V_j$ varying $i$ and $j$ (note that, each $V_i \times_{\mathbb{A}^1_W} V_j$ is non-empty since $W$ is irreducible) and for each irreducible component $V_0$ of $V^{\prime}$ which is also a connected component, there are dominant maps (\'etale maps) from $V_0$ to $V_i$ (for some $i$) induced by the projection maps $p_1$ and $p_2$. We have the following cases. \par
  %\begin{enumerate}
    % \item
\textbf{Case 1:} Suppose $\alpha \notin \overline{Im(\phi)}$. Consider the following diagram:  
$$\begin{tikzcd}[column sep=50pt, row sep=50 pt]
W^{\prime} \ar[r] \ar[d] &V \ar[r, "\phi"] \ar[d, "f" near start] &X \ar[d]\\
W \ar[r,"i_0" below] \ar[urr, dotted, "H(0)" near start] &\mathbb{A}^1_{W} \ar[r, "H" below] &\mathcal{S}^n(X)
\end{tikzcd}$$
where $i_0 : W \to \mathbb{A}^1_W$ is the $0$-section. Here the left square is cartesian and the lower triangle is commutative, since $H(0)$ factors through $X$. We have $\phi|_{W^{\prime}} \neq H(0)|_{W^{\prime}}$ as morphisms to $X$, since $\alpha \notin \overline{Im(\phi)}$. But they are the same in $\mathcal{S}^n(X)(W^{\prime})$. Suppose $m \geq 0$ is the least such that these two maps are the same in $\mathcal{S}^{m+1}(X)(W^{\prime})$. Thus there is a Nisnevich covering $W^{\prime \prime} \to W^{\prime}$ and there is a non-constant homotopy (by Remark \ref{useful remark}) $H^{\prime} \in \mathcal{S}^{m}(X)(\mathbb{A}^1_{W^{\prime \prime}})$, such that $H^{\prime}(0) = H(0)|_{W^{\prime \prime}}$. There is an irreducible component (say $W_0$) of $W^{\prime \prime}$ such that $H^{\prime}|_{\mathbb{A}^1_{W_0}}$ is non-constant. Since the map $W_0 \to W$ is dominant and $\alpha \in \overline{Im(H(0))}$, $\alpha \in \overline{Im(H^{\prime}|_{\mathbb{A}^1_{W_0}}(0))}$. \par
% \item
\textbf{Case 2:} Suppose $\alpha \in \overline{Im(\phi)}$. Moreover assume that there is an irreducible component (say $V_0$) of $V^{\prime}$ that maps to $V_i \times_{\mathbb{A}^1_W} V_j$ (for some $i$ and $j$) with $\alpha \in \overline{\phi(V_i)}$ and $p_1^*(\phi)|_{V_0} \neq p_2^*(\phi)|_{V_0}$. Then there is some $t$ such that $G_t|_{\mathbb{A}^1_{V_0}}$ is the required non-constant homotopy (if for each $t$, $G_t|_{\mathbb{A}^1_{V_0}}$ is constant, then $p_1^*(\phi)$ and $p_2^*(\phi)$ agree in $V_0$). Since the map $V_0 \to V_i$ is dominant, $\alpha \in \overline{Im(G_t|_{\mathbb{A}^1_{V_0}}(0))}$. In particular, if $\alpha \in \overline{\phi(V_i)}$ for every $i$, then we can take any irreducible component $V_0$ of $V^{\prime}$ such that $G_1|_{\mathbb{A}^1_{V_0}}$ is the non-constant homotopy. \par
 % \item
\textbf{Case 3:} Suppose $\alpha \in \overline{Im(\phi)}$ and there is a $j$ such that $\alpha \notin \overline{\phi(V_j)}$. If needed, renumbering $V_l$-s, we can assume that $\alpha \in \overline{\phi(V_1)}, \overline{\phi(V_2)}, \dots, \overline{\phi(V_i)}$ and $\alpha \notin \overline{\phi(V_{i+1})}, \dots, \overline{\phi(V_n)}$. Moreover we can assume that for each irreducible component $V_0$ of $V^{\prime}$ that maps to $V_m \times_{\mathbb{A}^1_W} V_l$ with $m \leq i$ we have, $p_1^*(\phi)|_{V_0} = p_2^*(\phi)|_{V_0}$ in $\mathcal{S}^{n-1}(X)(V_0)$. Otherwise the conclusion follows from Case 2. Thus we have for every $m \leq i$, 
\begin{gather*}
p_1^*(\phi)|_{V_m \times_{\mathbb{A}^1_W} V_l} = p_2^*(\phi)|_{V_m \times_{\mathbb{A}^1_W} V_l} \in \mathcal{S}^{n-1}(X)(V_m \times_{\mathbb{A}^1_W} V_l).
\end{gather*}
    Suppose there is a $t < n-1$ and there is an irreducible component $W_0$ of $V^{\prime}$ that maps to $V_m \times_{\mathbb{A}^1_W} V_l$ for some $m,l$ with $m \leq i$ such that $$p_1^*(\phi)|_{W_0} \neq p_2^*(\phi)|_{W_0} \in \mathcal{S}^{t}(X)(W_0).$$ Since $p_1^*(\phi)|_{W_0}$ and $p_2^*(\phi)|_{W_0}$ are the same in $\mathcal{S}^{n-1}(X)(W_0)$, we choose $t$ such that $p_1^*(\phi)|_{W_0}$ is same with $p_2^*(\phi)|_{W_0}$ in $\mathcal{S}^{t+1}(X)(W_0)$. Then there is a Nisnevich covering $V^{\prime \prime} \to W_0$ and a non-constant homotopy (by Remark \ref{useful remark}) $H^{\prime} \in \mathcal{S}^t(X)(\mathbb{A}^1_{V^{\prime \prime}})$ such that $H^{\prime}(0) = p_1^*(\phi)|_{V^{\prime \prime}}$. So there is an irreducible component $W^{\prime}_0$ of $V^{\prime \prime}$ such that $H^{\prime}|_{\mathbb{A}^1_{W^{\prime}_0}}$ is non-constant. Since the map $W_0^{\prime} \to V_m$ is dominant, we have $\alpha \in \overline{Im(H^{\prime}|_{\mathbb{A}^1_{W^{\prime}_0}})}$. \par 
    On the other hand, if there is no such $t$ then for every irreducible component $V_0$ of $V^{\prime}$ that maps to $V_m \times_{\mathbb{A}^1_W} V_l$ for some $m \leq i$, we have $p_1^*(\phi)|_{V_0} = p_2^*(\phi)|_{V_0}$ as morphisms to $X$. Therefore we have, 
\begin{gather*} 
p_1^*(\phi)|_{V_m \times_{\mathbb{A}^1_W} V_l} = p_2^*(\phi)|_{V_m \times_{\mathbb{A}^1_W} V_l}, \ \forall m \leq i \ \forall l
\end{gather*}
as morphisms to $X$. But then all $\overline{\phi(V_l)}$ are the same for every $l$, since $p_1: V_m \times_{\mathbb{A}^1_W} V_l \to V_m$ and $p_2: V_m \times_{\mathbb{A}^1_W} V_l \to V_l$ are dominant maps. It is a contradiction, since we have assumed there is some $j$ such that $\alpha \notin \overline{\phi(V_j)}$. \par
 %\end{enumerate}
Therefore, the proposition is proved.
\end{proof}
\begin{remark}
For any Nisnevich sheaf of sets $\mathcal{F}$, using the same argument as in the proof of Proposition \ref{track}, we have the following : Suppose there is a non-constant homotopy $H \in \mathcal{S}(\mathcal{F})(\mathbb{A}^1_W)$ for some $W \in Sm/k$ such that the image of $H(0)$ contains some $\alpha \in \mathcal{F}(Spec \ k)$. Then there is a non-constant homotopy $H^{\prime} \in \mathcal{F}(\mathbb{A}^1_{W^{\prime}})$ for some $W^{\prime} \in Sm/k$ such that the image of $H^{\prime}(0)$ contains $\alpha$.
\end{remark}

The next theorem is the main theorem of this section. It shows that being $\mathbb{A}^1$-connected gives $\mathbb{A}^1$-s in a variety.

\begin{theorem}\label{main}
Suppose $X \in Sm/k$ is $\mathbb{A}^1$-connected with $dim(X) \geq 2$ and $k$ is an algebraically closed field. Then one of the following holds :
\begin{enumerate} 
\item $\forall \ x \in X(k)$, there is a non-constant $\mathbb{A}^1$ through $x$.
\item  There is a non-constant homotopy $H : \mathbb{A}^1_Y \to X$, for some irreducible $Y \in Sm/k$, such that the dimension of the closure of the image of $H$ is at least $2$.
\end{enumerate}
 In particular for a surface $X \in Sm/k$, where $k$ is of characteristic zero, if $X$ is $\mathbb{A}^1$-connected, then $X$ is dominated by images of $\mathbb{A}^1$.
 \end{theorem}
 \begin{proof}
As $X$ is $\mathbb{A}^1$-connected, the sheaf $\mathcal{L}(X)$ is trivial \cite[Corollary 2.18]{bhs}. Suppose that $\exists \ \alpha \in X(k)$ such that there is no non-constant $\mathbb{A}^1$ through $\alpha$. Choose $\beta \in X(k)$ with $\beta \neq \alpha$. Also $\alpha \neq \beta \in \mathcal{S}(X)(Spec \ k)$, but $\alpha = \beta \in \mathcal{L}(X)(Spec \ k)$. Therefore, there is an $n \geq 1$ such that $\alpha = \beta \in \mathcal{S}^{n+1}(X)(Spec \ k)$ and $\alpha \neq \beta \in \mathcal{S}^n(X)(Spec \ k)$.  So there is a non-constant homotopy $H \in \mathcal{S}^{n}(X)(\mathbb{A}^1_k)$ such that $H(0)= \alpha$. Hence by applying Proposition \ref{track} repeatedly, there exists some $Y \in Sm/k$ irreducible, along with a non-constant homotopy $H^{\prime}: \mathbb{A}^1_Y \to X$, such that $\alpha \in \overline{Im(H^{\prime}(0))}$. Since $k$ is algebraically closed, the $k$-rational points are dense, so $H^{\prime}(0) \neq H^{\prime}(1)$ at some $k$-rational point. Therefore the image of $H^{\prime}$ contains a non-constant $\mathbb{A}^1$ and we have $\overline{Im(H^{\prime})}$ contains $\alpha$. Therefore $\overline{Im(H^{\prime})}$ is of dimension at least $2$, as we have assumed that there is no non-constant $\mathbb{A}^1$ through $\alpha$.  \par
Since $H^{\prime}$ is a non-constant homotopy, by shrinking $Y$ we can assume that $H^{\prime}(0,y) \neq H^{\prime}(1, y), \ \forall \ y \in Y(k)$ and the dimension of the closure of image is at least $2$. Thus if $X$ is a surface, the map $H^{\prime}$ is dominant. So there is a non-empty open subset $U$ of $X$ such that $U$ is contained in the image of $H^{\prime}$. Each $u \in U(k)$ has the preimage $(t, y) \in \mathbb{A}^1_Y$ for some $k$-point $y$ in $Y$. Therefore, $u$ is in the image of $H^{\prime}|_{\mathbb{A}^1_k \times \{y\}}$. Thus $X$ is dominated by images of $\mathbb{A}^1$. 
\end{proof}
\begin{corollary}\label{A1}
Suppose $X \in Sm/k$ is $\mathbb{A}^1$-connected and $k$ is an algebraically closed field. Then there is a non-constant $\mathbb{A}^1$ in $X$.  
\end{corollary}
\begin{corollary}\label{A2}
Suppose $X \in Sm/k$ is $\mathbb{A}^1$-connected surface and $k$ is an uncountable algebraically closed field of characteristic zero. Then $\bar{\kappa}(X) = -\infty$.
\end{corollary}
\begin{proof}
This follows because of the equivalence of $(1)$ and $(2)$ in Definition \ref{several A1} in this case and \cite[Proposition 1]{iit}.
\end{proof}
%\begin{proof}
%It follows from Theorem \ref{main} since $k$-rational points corresponds to closed points which are dense in a scheme over $k$.
%\end{proof}
\begin{remark}
In Corollary \ref{A1}, the assumption that $k$ is an algebraically closed field, is necessary. The unit sphere $T$ in $\mathbb{A}^3_{\mathbb{R}}$ given by the equation $x^2 + y^2 + z^2 = 1$ is $\mathbb{A}^1$-connected (\cite[Theorem 4.3.4]{sawant}), however there is no non-constant $\mathbb{A}^1_{\mathbb{R}}$ in $T$. If $X \in Sm/k$ is an $\mathbb{A}^1$-connected surface and the base field is countable (e.g. $\bar{\mathbb{Q}}$), even though $X$ is dominated by the images of $\mathbb{A}^1$'s,  we don't know whether $X$ has negative logarithmic Kodaira dimension.
%Note that the the way we have proved, the conclusion of the Theorem \ref{main} holds if we assume there is $n$ such that $\mathcal{S}^n(X)(Spec \ k)$ is trivial.
\end{remark}

\section{Characterisation of Affine Space} \label{charac}
 In this section we give some characterisations of $\mathbb{A}^n_k$ for the cases $n=2,3$ and $4$ using $\mathbb{A}^1$-homotopy theory. We give here the proof of the main theorem (Theorem \ref{main theorem}). A variant of the main theorem (Theorem \ref{main topology}) and its consequences are also given in this section. We thank the referee and Prof. Amartya Kumar Dutta for this version of Theorem \ref{main theorem}. \par
\begin{proof}[\textbf{Proof of the Theorem \ref{main theorem}}]
 One direction is clear by definition of $\mathbb{A}^1$-contractibility. Conversely, suppose $X$ is $\mathbb{A}^1$-contractible. We can consider $k$ as a subfield of an uncountable algebraically closed field $L$. As $L/k$ is the filtered colimit of its finitely generated sub-extensions over $k$, therefore by \cite[Corollary 1.24]{mv} the base change $X_L:=X \times_k L$ is $\mathbb{A}^1$-contractible. Then the Picard group of $X_L$ is trivial and the group of units of $X_L$ is $L^*$. Moreover by Corollary \ref{A2}, the logarithmic Kodaira dimension of $X_L$ is $-\infty$. Therefore, using \cite[Section 4.1]{miyanishi}, we get $X_L \cong \mathbb{A}^2_{L}$. Thus $\mathcal{O}(X)$ is an $\mathbb{A}^2$-form over $L/k$. 
%There is no non-trivial finitely generated $\mathbb{A}^2$-form over the field of characteristic zero. This follows from \cite[Theorem 3]{kamb}.  Therefore, we have $X \cong \mathbb{A}^2_k $. \par
Using \cite[Theorem 3]{kamb} one can show that $\mathcal{O}(X)$ is a trivial $\mathbb{A}^2$-form as $k$ is of characteristic $0$. Indeed, 
%Suppose, a finitely generated $k$-algebra $A$ is an $\mathbb{A}^2$-form over the field extension $L/k$. 
since $\mathcal{O}(X)$ is a finitely generated $k$-algebra, we can assume $L$ to be a finitely generated field extension over $k$. In characteristic zero any algebraic extension is separable, therefore after the base change we can assume $L$ to be a finitely generated purely transcendental extension (say, $L = k(X_1, X_2, .., X_n)$) over $k$ by \cite[Theorem 3]{kamb}. Again since $\mathcal{O}(X)$ is a finitely generated $k$-algebra, there is some $f \in k[X_1, X_2, .., X_n]$ such that $\mathcal{O}(X) \otimes_k k[X_1, X_2, .., X_n]_f \cong k[X_1, X_2, .., X_n]_f [X, Y]$. Taking quotient by some maximal ideal, we can assume $L$ is a separable algebraic extension. Therefore by \cite[Theorem 3]{kamb}, we get $\mathcal{O}(X) \cong k[X, Y]$. 
\end{proof}
If the base field is $\mathbb{C}$ and $X(\mathbb{C})$ is topologically contractible, then $X$ has trivial Picard group and trivial group of units, the same proof gives another characterisation of the affine complex plane.
\begin{theorem} \label{main topology}
A smooth complex surface $X$ is isomorphic to $\mathbb{A}^2_{\mathbb{C}}$ if and only if it is topologically contractible and $\mathbb{A}^1$-connected.
\end{theorem}

%\begin{remark}
%The above theorem is true for any algebraically closed field of characteristic zero. $\mathbb{A}^1$-connected affine surfaces are $\mathbb{A}^1$-uniruled by Theorem \ref{main}. So it has negative logarithmic Kodaira dimension \cite[Lemma 5.11]{km}. Therefore it is $\mathbb{A}^1$-ruled \cite[Theorem 0.1]{mist}. Moreover $\mathcal{O}(X)$ is U.F.D with units are trivial since $X$ is $\mathbb{A}^1$-contractible. Hence $X$ is isomorphic to $\mathbb{A}^2_k$ by Remark \ref{slice} and  \cite[Theorem 1]{mi}.
%\end{remark}
\begin{corollary} \label{stronger}
Suppose $X$ is a smooth complex surface which is topologically contractible and of non-negative logarithmic Kodaira dimension. Then $X$ is not $\mathbb{A}^1$-connected. 
%However $X$ is motivically contractible i.e. $M(X) \cong M(Spec \ \mathbb{C})$ in $\mathbf{DM_{gm}}(\mathbb{C}, \mathbb{Z})$ \cite[Theorem 1]{asok}. For example, the Ramanujam surface \cite[\S 3]{ra} and the tom Dieck-Petrie surfaces \cite[Theorem A]{dp} and all the non trivial homology planes are not $\mathbb{A}^1$-connected but motivically contractible. 
\end{corollary}
\begin{remark}
There are topological contractible complex surfaces which are affine modifications of $\mathbb{A}^2_{\mathbb{C}}$, but they are not $\mathbb{A}^1$-contractible. The tom Dieck-Petrie surfaces are the affine modifications of $\mathbb{A}^2_{\mathbb{C}}$ \cite[Example 3.1]{kaza} and they are topologically contractible \cite[Theorem A]{dp}. However the tom Dieck-Petrie surfaces are not even $\mathbb{A}^1$-connected (Corollary \ref{stronger}).
\end{remark}

\begin{corollary} \label{koras-russell}
Any Koras-Russell threefolds of the first kind over a field of characteristic zero cannot be the product of two proper subvarieties.
\end{corollary}
\begin{proof}
If possible, $X$ is a Koras-Russell threefold of first kind and $X$ is isomorphic to $Y \times_k Z$ where $Y$ and $Z$ are proper subvarieties of $X$. Then both $Y$ and $Z$ are smooth affine varieties. We can assume that $Y$ is a curve and $Z$ is a surface. Since $X$ is $\mathbb{A}^1$-contractible \cite[Theorem 1.1]{df}, being retract of $X$, both $Y$ and $Z$ are $\mathbb{A}^1$-contractible. Therefore, $Y \cong \mathbb{A}^1_k$ \cite[Claim 5.7]{ad} and $Z\cong \mathbb{A}^2_k$ (Theorem \ref{main theorem}). Thus $X$ is isomorphic to $\mathbb{A}^3_k$ which is a contradiction \cite[Thoerem 9.9]{fr}.
\end{proof}
Corollary \ref{koras-russell} holds in a more general setting:  Any Koras-Russell threefold (of any kind) cannot be the product of two other varieties. This can be proved using the properties of $\mathbb{G}_a$-actions without using $\mathbb{A}^1$-homotopy theory. This was pointed by the referee.

\begin{corollary}[Generalised Zariski's cancellation] \label{zar can}
Let $X$ and $Y$ be varieties over a field $k$ of charateristic zero. Suppose that $X$ is a surface and $X \times_k Y \cong \mathbb{A}^N_k$. Then $X \cong \mathbb{A}^2_k$.
\end{corollary}
\begin{proof}
If $X \times_k Y \cong \mathbb{A}^N_k$, then both $X$ and $Y$ are smooth affine $k$-varieties. Being retract of $\mathbb{A}^N_k$, $X$ is $\mathbb{A}^1$-contractible. Thus $X \cong \mathbb{A}^2_{k}$ by Theorem \ref{main theorem}.
\end{proof}
Theorem \ref{main theorem} has following immediate consequence in dimensions $3$ and $4$:
\begin{corollary} \label{4,5}
 An $\mathbb{A}^1$-contractible smooth affine threefold $X$ over a field of characteristic zero is isomorphic to $\mathbb{A}^3_{k}$ if and only if it is isomorphic to a product of two proper sub-varieties of lower dimension. Similarly an $\mathbb{A}^1$-contractible smooth affine fourfold $X$ over a field of characteristic zero is isomorphic to $\mathbb{A}^4_{k}$ if and only if it is isomorphic to a product of two proper sub-varieties each of dimension two.
\end{corollary}
The above corollaries (\ref{zar can}, \ref{4,5}) can be stated and proved without an appeal to $\mathbb{A}^1$-homotopy theory. The main ingredient here is the negativity of the logarithmic Kodaira dimension of the surfaces appearing in the proof , which we have derived from Theorem \ref{main}. \\\\

\textbf{Locally Nilpotent Derivation}

Let $k$ be a field of characteristic zero and $R$ is a $k$-algebra. The following definition is related to the property of being $\mathbb{A}^1$-ruled. 
\begin{definition} \cite[Section 1.1.7]{fr} 
A locally nilpotent $k$-derivation $D : R \to R$ is a $k$-linear derivation such that for each $a \in R$ $\exists \ n \in \mathbb{N}$ such that $D^{n}(a)=0$. The derivation $D$ has a slice if $\exists \ s \in R$  with $D(s)=1$. We denote the kernel of $D$ by $R^{D}$ which is a $k$-algebra and the set of all locally nilpotent $k$-derivations on $R$ will be denoted by $LND_k(R)$.
\end{definition}
%Makar-Limanov invariant of a $k$-domain $R$ is defined as, $ML(R):=\underset{D \in LND_{k}(R)}{\cap}R^D$. 
\par

Locally nilpotent derivation is an essential tool in affine algebraic geometry to characterise the polynomial rings. Miyanishi showed a two dimensional affine U.F.D. over an algebraically closed field $k$ with no non-trivial units is isomorphic to $k[x,y]$ if it admits a non-trivial locally nilpotent $k$-derivation. \cite[Theorem 1]{mi}. We want to also emphasise that locally nilpotent derivation corresponds to $\mathbb{G}_a$ -action only when $char(k) = 0$. 

\begin{remark}  \label{slice}
\begin{enumerate}
\item Suppose $D \in LND_k(R)$ has a slice $s \in R$. Then $R= R^{D}[s]$ i.e. $R$ is a polynomial ring over $R^{D}$ of one variable and $D=\frac{d}{d s}$, derivative with respect to $s$ \cite[Corollary 1.26]{fr}.
\item The locally nilpotent $k$-derivations on an affine $k$-domain $B$ correspond to the algebraic $\mathbb{G}_a$-actions on $Spec \ B$ \cite[Section 1.5]{fr}.
\item Let $X$ be an affine variety such that $\mathcal{O}(X)$ is a U.F.D. Then $X$ is $\mathbb{A}^1$-ruled if and only if there is a non-trivial locally nilpotent derivation on $\mathcal{O}(X)$ \cite[Proposition 2]{at} .

\end{enumerate}

\end{remark}
The fact that $\mathbb{A}^2_{k}$ is the only $\mathbb{A}^1$-contractible smooth affine surface over a field $k$ of characteristic zero has the following consequences.
\begin{corollary} 
A smooth affine threefold $X$ over a field of characteristic zero is isomorphic to $\mathbb{A}^3_{k}$ if and only if $X$ is $\mathbb{A}^1$-contractible and there exists a locally nilpotent derivation with a slice.

\end{corollary}

\begin{proof}
 One direction is straightforward. For the other direction, suppose $X$ is $\mathbb{A}^1$-contractible and there is a locally nilpotent derivation on $\mathcal{O}(X)$ with a slice. Then by Remark \ref{slice}, $X \cong U \times_k \mathbb{A}^1_k$, where $U$ is a smooth affine $k$-surface. The surface $U$ is $\mathbb{A}^1$-contractible, being a retract of $X$. Therefore by Theorem \ref{main theorem}, $U \cong \mathbb{A}^2_{k}$ and hence $X \cong \mathbb{A}^3_{k}$.

\end{proof}

In this context, there is an algebraic characterisation of the polynomial ring $k[x, y, z]$ where $k$ is an algebraically closed field of characteristic zero. A three dimensional finitely generated $k$-algebra, which is also a U.F.D., is isomorphic to $k[x, y, z]$ if and only if its Makar-Limanov invariant is trivial and it has a locally nilpotent derivation with a slice \cite[Theorem 4.6]{nn}.

\section{Regular functions on $\mathcal{S}(X)$}
So far we have used $\mathbb{A}^1$-homotopy theory to classify the affine varieties, specially the affine spaces. There are algebraic invariants associated to locally nilpotent derivation which also allow us to classify the affine spaces. One such invariant is the Makar-Limanov invariant. Makar-Limanov invariant of an affine variety $X$ is a subring $ML(X)$ of $\mathcal{O}(X)$. The ring $ML(X)$ is the set of regular functions on $X$, constant along the orbits of all $\mathbb{G}_a$-actions on $X$ \cite[Section 2.5]{fr},  affine variety $X$ over a characteristic $0$ field $k$. In particular for a $k$-domain $R$, 
$$ML(R):=\underset{D \in LND_{k}(R)}{\cap}R^D.$$ 
In case of $R = \mathcal{O}(X)$, we write $ML(X)$ instead of $ML(R)$. This invariant is not functorial. A two dimensional affine U.F.D. over an algebraically closed field $k$ of characteristic zero is isomorphic to $k[x,y]$ if and only if its Makar-Limanov invariant is trivial \cite[Theorem 9.12]{fr}. In this section, we construct a new functorial invariant $\mathcal{O}_{ch}(X)$ (Definition \ref{definition O_{ch}}) which is a subobject of $ML(X)$. We show that $\mathcal{O}_{ch}(X)$ is trivial i.e. $\mathcal{O}_{ch}(X) = k$ implies the existence of $\mathbb{A}^1$-s in $X$ (Theorem \ref{ochtrivial}).

Throughout the section we assume $k$ to be an algebraically closed field. Recall as in Section 4 by the phrase ``a line $g: \mathbb{A}^1 \to X$", we mean non-constant morphism $g: \mathbb{A}^1_k \to X$.
%Throughout the section we assume $k$ to be an algebraically closed field.
%In this section we define an invariant of a scheme which is functorial. To get functoriality we take those regular functions which fix all affine lines instead of a particular class of affine lines coming from a $\mathbb{G}_a$-action. This invariant will be a presheaf of $k$-algebras on $Sm/k$. We will see that it is homotopy invariant but it does not come from $\mathbb{A}^1$-homotopy category.
%\subsection{Definition and algebraic properties}
\begin{definition} \label{definition O_{ch}}
 Let $X$ be an affine $k$-variety and $\mathcal{O}(X)$ be the ring of regular functions on $X$. For a fixed $g: \mathbb{A}^{1}_{k} \to X$, define
\[
  \mathcal{O}_{ch,g}(X): =\left\lbrace  f \in \mathcal{O}(X) \;\middle|\;
   f \circ g \text{ is constant}
  \right\rbrace
\]
We define $\mathcal{O}_{ch}(X) = \underset {g \in Hom_{Sch/k}(\mathbb{A}^1_k, X)}{\bigcap} \mathcal{O}_{ch,g}(X)$, where $Sch/k$ is the category of finite type $k$-schemes.
 
%$\mathcal{O}_{ch}(X)$ to be the intersection of all $\mathcal{O}_{ch,g}(X)$-s where $g$ is varying as morphisms from $\mathbb{A}^1_k$ to $X$.
\end{definition}

 We get the following immediate properties of $\mathcal{O}_{ch,g}(X)$.
\begin{lemma} \label{constant}
Suppose $X$ is an affine $k$-variety and $g: \mathbb{A}^{1}_{k} \to X$ a $k$-morphism.
\begin{enumerate}
\item A morphism $\phi:\mathbb{A}^{1}_{k} \to \mathbb{A}^{1}_{k}$ is constant if and only if the induced $k$-algebra homomorphism $\Tilde{\phi}: k[T] \to k[T]$ takes $T$ to an element of $k$.
\item $\mathcal{O}_{ch,g}(X)$ is a $k$-subalgebra of $\mathcal{O}(X)$. In particular, $\mathcal{O}_{ch}(X)$ is a $k$-subalgebra of $\mathcal{O}(X)$.
\item Suppose $f_1, f_2 \in \mathcal{O}(X)$. If the product $f_{1}f_{2} \in \mathcal{O}_{ch,g}(X)$ is non-zero, then $f_{1} \in \mathcal{O}_{ch,g}(X)$ and $f_{2} \in \mathcal{O}_{ch,g}(X)$.
\item The group of units of $X$, $\mathcal{O}(X)^* \subset \mathcal{O}_{ch,g}(X)$. Thus $\mathcal{O}(X)^* \subset \mathcal{O}_{ch}(X)$. If $\mathcal{O}_{ch}(X)$ is trivial, then $X$ has trivial group of units.
\item Suppose, $X$ is of dimension at least two. Then the morphism $\bar{i}: X \to Spec(\mathcal{O}_{ch,g}(X))$ induced by the inclusion $i: \mathcal{O}_{ch,g}(X) \to \mathcal{O}(X)$ is birational.
\end{enumerate}
\end{lemma}

\begin{proof}
\textbf{(1),(2) and (3): }The proofs are quite straightforward. So we leave it to the reader. \par
%\item Since $f_{1}f_{2} \in \mathcal{O}_{ch,g}(X)$, so $\Tilde{g}(f_1f_2) \in k$  by Remark \ref{constant}. So either $\Tilde{g}(f_1)$ or $\Tilde{g}(f_2)$ is zero or both are constant polynomial in $k[T]$. Therefore $f_{1}$ or $f_{2}$ is in $\mathcal{O}_{ch,g}(X)$ by Remark \ref{constant} again.
\textbf{(4)}: Suppose, $f \in \mathcal{O}(X)^*$, then $f$ is a morphism from $X$ to $\mathbb{G}_m$. Therefore $f \circ g$ is constant, since $\mathbb{G}_m$ is $\mathbb{A}^1$-rigid (Remark \ref{Open-Closed}). Hence $f \in \mathcal{O}_{ch,g}(X)$. \par   
\textbf{(5):} Since $X$ is of dimension at least $2$, so $g:\mathbb{A}^{1}_{k} \to X$ is not dominant. The image of $g$ is closed in $X$. Indeed, if $g$ is non-constant, we extend $g$ to a morphism $\bar{g}: \mathbb{P}^1_k \to \overline{X}$ ($\overline{X}$ is a compactification of $X$). Since $\mathbb{P}^1_k$ is a projective variety and $X$ is an affine variety, there are no non-constant morphisms from $\mathbb{P}^1_k$ to $X$. Thus $\bar{g}$ maps the point at infinity of $\mathbb{P}^1_k$ to a point in $\overline{X} \setminus X$. The morphism $\bar{g}$ is the composition of two proper morphisms $\mathbb{P}^1_k \xrightarrow{\text{graph of }\bar{g}} \mathbb{P}^1_k \times_k \overline{X} \xrightarrow{\text{projection}} \overline{X}$, 
%therefore $\bar{g}$ is proper. %Since $\overline{X}$ is projective variety, both 
%are proper morphisms, 
so $\bar{g}$ is a proper morphism. 
Thus the image of $\bar{g}$ is closed and $Im(g) = Im(\bar{g}) \cap X$ is closed in $X$. Therefore,  image of $g$ is given by some ideal $I$ of $\mathcal{O}(X)$. So its complement is the union of basic open set $D(f)$-s, $f \in I$. Choose $f \in I$ with $D(f)$ is non-empty. Then $\Tilde{g}(f)=0$ (where $\Tilde{g}:\mathcal{O}(X) \to k[T]$ is induced by $g$.). So $fh + \mu \in \mathcal{O}_{ch,g}(X)$ $\forall \ h \in \mathcal{O}(X), \forall \mu \ \in k$ by Part (1). We have an injective homomorphism $i_*: \mathcal{O}_{ch,g}(X)_{f} \to \mathcal{O}(X)_{f}$ induced by the inclusion $i$. The map $i_*$ is also surjective. Indeed for $\frac{h}{f^{k}} \in \mathcal{O}(X)_{f}$, $fh \in \mathcal{O}_{ch,g}(X)$ and the element $\frac{fh}{f^{k+1}}$ is mapped to $\frac{h}{f^{k}}$. Hence $\mathcal{O}_{ch,g}(X)_{f}$ and $\mathcal{O}(X)_{f}$ are isomorphic and therefore $X$ and $Spec(\mathcal{O}_{ch,g}(X))$ are birational.
\end{proof}
%\begin{proof}
%If $\phi(T) \in F$, then for a prime ideal $P$ of $F[T]$, $P \cap F = (0)$. So $\phi^{-1}(P)=\phi^{-1}((0))$, hence $f$ is constant. Conversely if possible assume $\phi(T)$ is a non-constant polynomial then only zero polynomial is mapped to zero, so $\phi$ is one-one. On the other hand $\phi(T) \in \mathfrak{m}$ for some maximal ideal $\mathfrak{m}$. So $\phi ^{-1}(\mathfrak{m})$ is not a zero ideal but $\phi$ is one-one. Hence $f$ is non-constant.
%\end{proof}
%\begin{proof}
%Any constant function in $\mathcal{O}(X)$ is in $\mathcal{O}_{ch,g}(X)$. Suppose $f_{1}, f_{2} \in \mathcal{O}_{ch,g}(X)$. Then $\Tilde{g}(f_1), \Tilde{g}(f_2) \in k$ ($\Tilde{g}:\mathcal{O}(X) \to k[T]$ is induced by $g$)  by Remark \ref{constant}. Therefore $\Tilde{g}(f_{1}+f_{2}) $ and $\Tilde{g}(f_{1}f_{2})$ are in $k$. Hence $(f_{1}+f_{2}) \circ g$ and $(f_{1}f_{2})\circ g$ are constant by Remark \ref{constant}. Therefore $f_{1}+f_{2}, f_{1}f_{2} \in \mathcal{O}_{ch, g}(X)$. So it is a $k$-subalgebra of $\mathcal{O}(X)$. 
%\end{proof}
%The following proposition says that $Spec(\mathcal{O}_{ch,g}(X))$ is obtained from $X$ by collapsing the closure of $g$.
%\begin{remark}
%The above proof shows that the map $X$ to $Spec(\mathcal{O}_{ch,g}(X))$ is isomorphism restricted to each basic open set $D(f)$, $f \in I$, varying $f \in I$ we get, the open subscheme complement of closure $g(\mathbb{A}^{1}_k)$ in $X$ is isomorphic to its image. 
%\end{remark}
\begin{remark}
In the above Lemma \ref{constant}(5), the assumption about the dimension of $X$  is necessary. If $X=\mathbb{A}^{1}_{k}$ and $g$  be the identity map on $\mathbb{A}^{1}_{k}$, then $\mathcal{O}_{ch,g}(X)$ is trivial. Property $(3)$ of $\mathcal{O}_{ch,g}(X)$ in the above Lemma \ref{constant} is similar to a ring being factorially closed which is satisfied by kernels of a locally nilpotent derivation hence by the Makar-Limanov invariant \cite[Section 1.4, Principle 1]{fr}. Also observe that, $\mathcal{O}_{ch,g}(X)$ may not always be finitely generated $k$-subalgebra of $\mathcal{O}(X)$. For instance, suppose $X= \mathbb{A}^2_k$ and $g$ is the $y$-axis. Then $\mathcal{O}_{ch,g}(X) = k + xk[x,y]$. This subring of $k[x,y]$ is not Noetherian. For this, consider the chain of ideals $\{I_n\}_n$ in $\mathcal{O}_{ch,g}(X)$: $I_n$ is the ideal generated by $\{x, xy,xy^2,.., xy^{n-1}\}$. This chain of ideals does not stabilize. 
\end{remark}
%\subsection{Geometric Properties}
\begin{remark}
We can describe $\mathcal{O}_{ch,g}(X)$ explicitly. The image of the affine line $g: \mathbb{A}^1_k \to X$ is closed in $X$. Let  $\Tilde{g}:\mathcal{O}(X) \to k[T]$ be the $k$-algebra homomorphism induced by $g$. A regular function on $X$ is in $\mathcal{O}_{ch,g}(X)$ if and only if its image is in $k$ under $\Tilde{g}$. Thus for $\phi \in\mathcal{O}_{ch, g}(X) $, $\phi - \Tilde{g}(\phi) \in Ker(\Tilde{g})$. Therefore, $\mathcal{O}_{ch, g}(X) = k + Ker(\Tilde{g})$. If $\phi = \lambda + \theta$ for some constant $\lambda$ and $\theta \in Ker(\Tilde{g})$, then $\phi$ takes value $\lambda$ along $Im(g)$. \\
\textbf{Suppose} $ \mathbf{g_1}$ \textbf{ and } $\mathbf{g_2}$ \textbf{ are  two intersecting } $\mathbf{\mathbb{A}^1}$ \textbf{-s in } $\mathbf{X}$ \textbf{: } \par
 If $\phi \in \mathcal{O}_{ch, g_1}(X) \cap \mathcal{O}_{ch, g_2}(X) $, then $\phi = \lambda + \theta = \lambda^{\prime} + \theta^{\prime}$ for some constants $\lambda, \lambda^{\prime}$ and $\theta, \theta^{\prime}$ are in kernel of $\Tilde{g_1}$ and $\Tilde{g_2}$ respectively. Since $g_1$ and $g_2$ intersect, $\lambda = \lambda^{\prime}$. Therefore $ \mathcal{O}_{ch, g_1}(X) \cap \mathcal{O}_{ch, g_2}(X) = k + (Ker(\Tilde{g_1}) \cap Ker(\Tilde{g_2}))$, if $g_1$ and $g_2$ intersect. \\ 
\textbf{Suppose } $\mathbf{g_1}$ \textbf{ and } $\mathbf{g_2}$ \textbf{ are parallel : } \par
If the images of $g_1$ and $g_2$ are disjoint, then $\mathcal{O}_{ch, g_1}(X) \cap \mathcal{O}_{ch, g_2}(X)$ properly contains $k + (Ker(\Tilde{g_1}) \cap Ker(\Tilde{g_2}))$ from the following lemma (Lemma \ref{geometric properties}). 
\end{remark}
%  Hence we have the following:
%This description gives the following lemma (Lemma \ref{geometric properties}). 
\begin{definition}
A line $g:\mathbb{A}^1 \to X$ is called isolated if it does not intersect any other lines in $X$ i.e. for any line $h: \mathbb{A}^1 \to X$, if $Im(h) \neq Im(g)$ then $Im(h) \cap Im(g) = \emptyset$.
\end{definition}
\begin{lemma} \label{geometric properties}
Suppose $X$ is an affine $k$-variety.
\begin{enumerate}
\item Suppose $g_1, g_2,..,g_n$ are pairwise parallel lines in $X$ (i.e. $Im(g_i) \cap Im(g_j) = \emptyset, \ \forall i \neq j$) and $c_1, c_2,.., c_n$ are $n$ many constants. Then there is $f \in \mathcal{O}(X)$ such that $f = c_i$ along $Im(g_i)$.
\item Let $X$ be a smooth affine surface such that $\mathcal{O}(X)$ is a U.F.D. Suppose there is a line $g: \mathbb{A}^1 \to X$ which is isolated. Then $\mathcal{O}_{ch}(X)$ is non-trivial.
\end{enumerate}
\end{lemma}
\begin{proof}
\begin{enumerate}
\item Since $g_i$ and $g_j$ are parallel, $Ker(\Tilde{g_i}) + Ker(\Tilde{g_j}) = \mathcal{O}(X)$. Indeed, if $Ker(\Tilde{g_i}) + Ker(\Tilde{g_j})$ is contained in some maximal ideal of $\mathcal{O}(X)$, then there is a common point of $g_1$ and $g_2$. So the ideals $Ker(\Tilde{g_i})$ and $Ker(\Tilde{g_j})$ are pairwise comaximal. Thus by Chinese remainder theorem, there exists $f \in \mathcal{O}(X)$ such that $f$ is $c_i$ along $Im(g_i)$. 
\item Since $\mathcal{O}(X)$ is a U.F.D., there is a $f \in \mathcal{O}(X)$ irreducible such that the zero set of $f$ is the closed set $Im(g)$. Then $f$ is non-zero in the complement of $Im(g)$. So for any other line $h: \mathbb{A}^1 \to X$ with $Im(h) \neq Im(g)$, $f$ is everywhere non-zero along $Im(h)$ as $g$ is an isolated line. Hence $f$ must be constant along $Im(h)$, since $\mathbb{G}_m$ is $\mathbb{A}^1$-rigid (Remark \ref{Open-Closed}). Therefore $f$ is a non-trivial element in $\mathcal{O}_{ch}(X)$.
\end{enumerate}
\end{proof}
 Therefore, for a smooth affine surface $X$ with trivial Picard group, if  $\mathcal{O}_{ch}(X)$ is trivial, then all lines in $X$ cannot be parallel to each other (by ``all lines in $X$ are parallel to each other", we mean given any two lines $g_1, g_2 : \mathbb{A}^1  \to X$, we have $Im(g_1) \cap Im(g_2) = \emptyset$). Note that in $\mathbb{A}^1 \times \mathbb{G}_m$, any line parallel to $x$-axis is an isolated line and any polynomial of $y$ is in $\mathcal{O}_{ch}(\mathbb{A}^1 \times \mathbb{G}_m)$.
\begin{definition}
A chain connected component of $X$ is defined to be the largest subset of $X(k)$ such that any two points in it can be joined by a chain of $\mathbb{A}^1$-s.
\end{definition}
\begin{remark}
Let $T \subset X(k)$ be  a chain component. then T is the union of lines in $X$ such that the points in the images of the lines are in $T$.
\end{remark}
\begin{theorem} \label{ochtrivial}
Let $X$ be a smooth affine surface such that $\mathcal{O}(X)$ is a U.F.D. Then $\mathcal{O}_{ch}(X)$ is trivial if and only if there is some dense chain connected component of $X$.
\end{theorem}
\begin{proof}
Suppose there is some chain connected component of $X$ which is dense in $X$. Let $T$ be the union of all lines in that chain connected component and $f$ is in $\mathcal{O}_{ch}(X)$. The function $f$ is constant along $T$, since any two points of $T$ can be joined by chain of lines. But $T$ is dense in $X$. Therefore $f$ is constant. \par
           On the other hand, assume that $\mathcal{O}_{ch}(X)$ is trivial. If possible, there is a chain component (say $T$) which is the union of finitely many lines (i.e. finitely many distinct images). Then it is closed. There is some $f \in \mathcal{O}(X)$ such that its zero set is $T$. Then $f$ is non-zero along every other line outside $T$. Therefore it is constant along each line outside $T$. But $f$ is non-constant. This gives a contradiction. Therefore every chain connected component is a union of infinitely many lines (i.e. infinitely many distinct images). Choose any such chain connected component, say $S$. Its closure cannot be of dimension $1$, since it contains infinitely many lines, hence $S$ is dense in $X$.
\end{proof}
%\begin{remark} \label{ochunit}
%Suppose $\mathcal{O}_{ch}(X)$ is trivial. Then $\mathcal{O}(X)^* = k^*$. Indeed, if $f \in \mathcal{O}(X)^*$, then $f$ is a morphism from $X$ to $\mathbb{G}_m$. Therefore $f$ is constant along each line in $X$, since $\mathbb{G}_m$ is $\mathbb{A}^1$-rigid (Remark \ref{Open-Closed}). Hence $f \in \mathcal{O}_{ch}(X)$. So $f$ is constant. 
%\end{remark}
%\subsection{Makar-Limanov invariant}

%It is a useful invariant in classification of algebraic varieties from algebraic viewpoint. For instance over an algebraically closed field $k$ of characteristic zero a UFD of transcendence degree $2$ over $k$ is a polynomial ring if its Makar-Limanov invariant is trivial \cite{fr}[Theorem $9.9$]. In the connection of \ref{main theorem} we have a similar characterisation of $k[x,y,z]$: over an algebraically closed field of characteristic zero a three dimensional affine domain which is a UFD is a polynomial ring if its Makar-Limanov invariant is trivial and it has a locally nilpotent derivation with slice \cite{nn}[Theorem $3.8$]. \par
%Unfortunately Makar-Limanov invariant is not functorial. Therefore we are looking of an invariant which is functorial and is related to Makar-Limanov invariant.

%\subsection{Makar-Limanov invariant and $\mathcal{O}_{ch}(-)$}
%As we have mentioned in beginning our goal is to obtain an invariant which is functorial. 

\textbf{Functoriality}

Suppose, $\alpha: Y \to X$ is a morphism of affine $k$-varieties. For an affine line $g$ in $Y$, $\alpha \circ g$ is an affine line in $X$. So $f \circ (\alpha \circ g)$ is constant if $f \in \mathcal{O}_{ch}(X)$. Thus the morphism $\alpha$ induces $\alpha ^{*}: \mathcal{O}(X) \to \mathcal{O}(Y)$ that restricts to a $k$-algebra homomorphism $\mathcal{O}_{ch}(X) \to \mathcal{O}_{ch}(Y)$. Therefore, $\mathcal{O}_{ch}(X)$ is functorial in $X$. 
%\begin{proposition}
%If $X$ and $Y$ are isomorphic $k$-schemes then $\mathcal{O}_{ch}(X)$ and $\mathcal{O}_{ch}(Y)$ are isomorphic $k$-algebras.
%\end{proposition}
\begin{proposition} \label{trivial}
Let $k$ be a field of characteristic $0$. Suppose $X$ is an affine $k$-variety. Then $\mathcal{O}_{ch}(X) \subset ML(X)$. Therefore, if $ML(X)$ is trivial then $\mathcal{O}_{ch}(X)$ is trivial.
\end{proposition}
\begin{proof}
Suppose $f \in \mathcal{O}_{ch}(X)$ and $D$ is a locally nilpotent $k$-derivation on $\mathcal{O}(X)$. We need to show that $f \in Ker(D)$. We have a $k$-algebra homomorphism $exp(D): \mathcal{O}(X) \to \mathcal{O}(X)[T]$ defined as $exp(D)(g)= \sum_{n=0}^{\infty} \frac{D^{n}(g)}{n!} T^{n}$. Fix $x \in X(k)$. Consider the following composition of $k$-algebra homomorphisms (this is precisely considering $f$ along each of the orbits of $x \in X(k)$): 
$$k[T] \xrightarrow{T \mapsto f} \mathcal{O}(X) \xrightarrow{exp(D)} \mathcal{O}(X)[T] \xrightarrow{\text{evaluation at } x} k[T]$$
This composition takes $T$ to an element of $k$ by Lemma \ref{constant}. Suppose, $exp(D)(f) = \sum_{n=0}^{\infty} f_n T^n \in \mathcal{O}(X)[T]$, where $f_n \in \mathcal{O}(X)$. Then for every $n \geq 1$, $f_n(x) = 0, \ \forall x \in X(k)$. Since $X(k)$ is dense in $X$, $f_n$-s are zero for every $n \geq 1$. Thus $exp(D)(f)=f$. Hence $f \in Ker(D)$.
%\begin{lemma}
%$D$ is a locally nilpotent $k$-derivation on a $k$-algebra $R$. Then 
%\[
  %Ker(D) =\left\lbrace  f \in R \;\middle|\;
 %exp(D)(f)=f \right\rbrace
%\]
%\end{lemma}
%The orbits of a $\mathbb{G}_{a}$-action on $X$ are points or affine lines in $X$. So a regular function on $X$ which fixes all affine lines in $X$ is constant in each orbits. Now the proof follows from proposition \ref{corres} .
\end{proof}
 \begin{proposition}
 $\mathcal{O}_{ch}(-)$ is homotopy invariant i.e. $\mathcal{O}_{ch}(X) = \mathcal{O}_{ch}(X \times \mathbb{A}^1_k)$ (both are $k$-subalgebras of $\mathcal{O}(X \times \mathbb{A}^1_k)$), $\forall X \in Sm/k$.
 \end{proposition}
 \begin{proof}
 %The morphism $$Id \times i_0:Sing_*^{\mathbb{A}^1}(X) \to Sing_*^{\mathbb{A}^1}(X \times \mathbb{A}^1_k)$$ is a simplicial homotopy equivalence \cite[Corollary 3.5]{mv}.
%hence it is a simplicial weak equivalence. So in particular the presheaves $\pi_0(Sing^{\mathbb{A}^1}(X))$ and $\pi_0(Sing^{\mathbb{A}^1}(X \times \mathbb{A}^1_k))$ are isomorphic. 
%Therefore the corollary follows from the Proposition \ref{reg}. \\
The projection map $p: X \times_{k} \mathbb{A}^1_k \to X$ induces an injective homomorphism $p^*: \mathcal{O}_{ch}(X) \to \mathcal{O}_{ch}(X \times \mathbb{A}^1_k)$. For the surjectivity, suppose $f \in \mathcal{O}_{ch}(X \times_k \mathbb{A}^1_k)$. For each $x \in X(k)$, consider the line $j_x: \mathbb{A}^1_k \to X \times_k \mathbb{A}^1_k$ defined as $t \mapsto (x,t)$. As $f \in \mathcal{O}_{ch}(X \times_k \mathbb{A}^1_k)$, $f \circ j_x$ is constant. Thus for every $x \in X(k)$ and $t \in \mathbb{A}^1_k(k)$, we have $f \circ i_0 \circ p(x, t) = f(x, t)$ (where $i_0: X \to X \times_k \mathbb{A}^1_k$ is the $0$-section). Since $X(k)$ is dense in $X$, $f = f \circ i_0 \circ p$. So $p^*(f \circ i_0) = f$, which proves the surjectivity of $p^*$.
%since if $g$ is an affine line in $X \times \mathbb{A}^{1}_k$ then $p \circ g$ is an affine line in $X$ where $p: X \times \mathbb{A}^{1}_{k} \to X$ is the projection map. Conversely assume $f \in \mathcal{O}_{ch}(X \times \mathbb{A}^{1}_{k})$ and $f= \sum_{i=0}^{d}f_{i}T^{i} \in \mathcal{O}(X)[T]$. First we will show $f \in \mathcal{O}(X)$ i.e. $f=f_{0}$. Fix a point $x \in X$, $f$ is constant on the affine line $\{x\} \times \mathbb{A}^{1}_{k}$. So the composition $$k[T] \xrightarrow{T \mapsto f} \mathcal{O}(X)[T] \xrightarrow{\text{evaluation at } x} k[T]$$ takes $T$ to a element of $k$ by lemma $\ref{constant}$ hence each $f_{i}$ is zero at each point of $X$, $\forall i \geq 1$. Hence $f = f_{0} \in \mathcal{O}(X)$. Now we will show $f \in \mathcal{O}_{ch}(X)$. Suppose $g:\mathbb{A}^{1}_{F} \to X$ is an affine line in $X$. Compose $g$ with zero section $\theta: X \to X \times \mathbb{A}^{1}_{k}$, $\theta \circ g$ is an affine line in $X \times \mathbb{A}^{1}_{k}$ so $f \circ (\theta \circ g)$ is constant. So $f \circ g$ is constant by lemma \ref{constant}. Here what we have is that the inclusion $\mathcal{O}(X) \hookrightarrow \mathcal{O}(X)[T])$ restricts from $\mathcal{O}_{ch}(X)$ to $\mathcal{O}_{ch}(X \times \mathbb{A}^{1}_{k})$ is the identity map. Thus $\mathcal{O}_{ch}(X)=\mathcal{O}_{ch}(X \times \mathbb{A}^{1}_{k})$.
 \end{proof}
 \begin{remark}
Makar-Limanov invariant is not homotopy invariant. If $X$ is a Koras-Russell threefold of the first kind over $\mathbb{C}$, then $ML(X)=\mathbb{C}[T]$ (\cite[Theorem 9.9]{fr}) and $ML(X \times \mathbb{A}^{1}_{\mathbb{C}})=\mathbb{C}$ (\cite[Section 1]{dou}). From Proposition \ref{trivial} and homotopy invariance of $\mathcal{O}_{ch}(X)$, we have $\mathcal{O}_{ch}(X)=\mathcal{O}_{ch}(X \times \mathbb{A}^{1}_{\mathbb{C}})=\mathbb{C}$.
\end{remark}
\begin{remark} \label{subalgebra}
For an affine variety $X \in Sm/k$, recall the sheaf $\mathcal{S}(X)$ of $\mathbb{A}^1$-chain connected components of $X$ is the coequalizer of two morphisms 
$$\begin{tikzcd}
\underline{Hom}(\mathbb{A}^1_k, X) \ar[r, shift left=.75ex, "\theta_0"] \ar[r, shift right=.75ex,swap, "\theta_1"] 
& X
\end{tikzcd}$$
in $Shv(Sm/k)$ (Remark \ref{propertiesS(X)}). There is a canonical isomorphism 
$$Hom_{Sch/k}(X, \mathbb{A}^1_k) \cong \mathcal{O}(X).$$ 
The natural epimorphism $\pi: X \to \mathcal{S}(X)$ induces a monomorphism 
  $$Hom_{Shv(Sm/k)}(\mathcal{S}(X), \mathbb{A}^1_k) \to Hom_{Sm/k}(X, \mathbb{A}^1_k).$$ 
This way we identify $Hom_{Shv(Sm/k)}(\mathcal{S}(X), \mathbb{A}^1_k)$ as a $k$-subalgebra of $\mathcal{O}(X)$. 
%Thus a morphism from $\mathcal{S}(X)$ to $\mathbb{A}^1_k$ gives an element in $\mathcal{O}(X)$. 
%We describe $\mathcal{O}_{ch}(X)$ as morphisms from $\mathcal{S}(X)$ to $\mathbb{A}^1_k$.
\end{remark} 
\begin{proposition} \label{reg}
As $k$-subalgebras of $O(X)$, we have $$\mathcal{O}_{ch}(X) = Hom_{Shv(Sm/k)}(\mathcal{S}(X), \mathbb{A}^1_k).$$
\end{proposition}
\begin{proof}
%$\pi_0(Sing^{\mathbb{A}^1}(X)$ is the coequaliser of two maps from $\underline{Hom}(U \times \mathbb{A}^1)$ to $Hom(U,X)$ induced by $0$-section and $1$-section. Since $\mathbb{A}^1$ is a sheaf on Nisnevich site, so the morphisms from $\pi_0^{ch}(X)$ to $\mathbb{A}^1$ correspond to morphism from the presheaf $\pi_0(Sing^{\mathbb{A}^1}(X))$ to $\mathbb{A}^1_k$. Suppose $\phi \in \mathcal{O}_{ch}(X)$. 
Suppose $\phi \in \mathcal{O}_{ch}(X)$, $\phi$ gives a morphism $X$ to $\mathbb{A}^1_k$. We will show that $\phi \circ \theta_0 = \phi \circ \theta_1$ as morphisms from $\underline{Hom}(\mathbb{A}^1_k, X)$ to $\mathbb{A}^1_k$ i.e. to show that $\forall \ U \in Sm/k$ and $f: \mathbb{A}^1_U:= \mathbb{A}^1_k \times_k U \to X$, $\phi \circ f \circ \sigma_0 = \phi \circ f \circ \sigma_1$ where $\sigma_0, \sigma_1: U \to \mathbb{A}^1_U$ are the $0$-section and the $1$-section respectively. For any $x \in U(k)$, consider the morphism $i_x: \mathbb{A}^1_k \to \mathbb{A}^1_U$ defined as $t \mapsto (t,x)$. Composing $f$ with $i_x$, we get a morphism from $\mathbb{A}^1_k$ to $X$. Since $\phi \in \mathcal{O}_{ch}(X)$, $\phi \circ f \circ i_x$ is constant. Thus $\phi(f(\sigma_0(x))) = \phi(f(\sigma_1(x)))$. The $k$-points are dense in $U$, so $\phi \circ f \circ \sigma_0 = \phi \circ f \circ \sigma_1$. Therefore, $\phi$ induces a unique morphism of sheaves from $\mathcal{S}(X)$ to $\mathbb{A}^1_k$. Thus using the identification we have observed $Hom_{Shv(Sm/k)}(\mathcal{S}(X), \mathbb{A}^1_k)$ as $k$-subalgebra of $\mathcal{O}(X)$, we have $\phi \in Hom_{Shv(Sm/k)}(\mathcal{S}(X), \mathbb{A}^1_k)$. \par
%Now define a function $\Psi:\mathcal{O}_{ch}(X) \to Hom_{Shv(Sm/k)}(\pi_0^{ch}(X)), \mathbb{A}^1_k)$ as $\phi \mapsto \Phi$. 
%Now to check that it is well-defined. Suppose $f,g: U \to X$ such that $\exists \eta : \mathbb{A}^1_U \to X$ with $\eta(0)=f$ and $\eta(1)=g$. 
%$\Psi$ is one-one since the $\phi$ is the composition of $\Psi(\phi)$ with the natural morphism $X$ to $\pi_0^{ch}(X)$.
Conversely suppose, $\eta \in Hom_{Shv(Sm/k)}(\mathcal{S}(X), \mathbb{A}^1_k)$ and $g :\mathbb{A}^1_k \to X$ is a morphism . Then $\eta \circ \pi$ gives a morphism from $X$ to $\mathbb{A}^1_k$.  The morphism $g$ is homotopic to the constant map. Indeed, there is a homotopy $H: \mathbb{A}^1_k \times \mathbb{A}^1_k \to X$ as the composition of $g$ with the multiplication map $\mathbb{A}^1_k \times \mathbb{A}^1_k \to \mathbb{A}^1_k$ ($(s,t) \mapsto st$). Then $H \sigma_0$ is the constant map and $H \sigma_1 = g$ , where $\sigma_0, \sigma_1: \mathbb{A}^1_k \to \mathbb{A}^1_k \times \mathbb{A}^1_k$ are the $0$-section and the $1$-section (where we put $0$ and $1$ in the first coordinate) respectively. Since $\mathcal{S}(X)$ is the coequaliser of $\theta_0$ and $\theta_1$ (Remark \ref{propertiesS(X)}, (1)), we have 
$$\eta \circ \pi \circ H \circ \sigma_0 = \eta \circ \pi \circ H \circ \sigma_1$$ 
So $\eta \circ \pi \circ g$ is constant. This implies $\eta \circ \pi \in  \mathcal{O}_{ch}(X)$.  Now using the identification in remark \ref{subalgebra} we get $\eta \in \mathcal{O}_{ch}(X)$.
   % Now to show $\Psi$ is surjective. Suppose we have a morphism $\Theta$ from $\pi_0^{ch}(X))$ to $\mathbb{A}^1_k$. compose it with natural map $X$ to $\pi_0^{ch}(X)$, we have a morphism $\psi$ from $X$ to $\mathbb{A}^1_k$. Now we claim that $\psi \in \mathcal{O}_{ch}(X)$. Suppose $\theta: \mathbb{A}^1_F \to X$. Choose $\alpha \in \mathbb{A}^1_F$, $F$-rational point and $\beta$ is a closed point of $\mathbb{A}^1_F$. We will show $\psi \circ \theta(\alpha) = \psi \circ \theta(\beta)$. Consider the composite field extension of $k$ is $E=k(\beta)F$. Now $\beta$ is given by some irreducible polynomial in $F[T]$ which has a root in $E$, say $a$. Consider the morphism $\mu$ from $\mathbb{A}^1_E$ to $\mathbb{A}^1_F$ given by $T$ goes to $aT$. Then $\mu(0)=\alpha$ and $\mu(1)=\beta$. So $\psi \circ \theta(\alpha) = \psi \circ \theta(\beta)$ since $\psi$ factors through natural map from $X$ to coequaliser. Hence $\Psi$ is a bijection.  
\end{proof}
%\begin{lemma}
%$\mathcal{O}_{ch}(X)$ is the set of all regular functions on $\pi_0(Sing^{\mathbb{A}^{1}}(X))$.
%\end{lemma}
%\begin{proof}
%If $x, y \in X$ are in same component of $Sing^{\mathbb{A}^1}(X)(Spec k)$ then there is a morphism $f: \mathbb{A}^1_k \to X$ such that $x$ and $y$ are in image of $f$. Since $f \in O_{ch}(X)$, $f(x)=f(y)$. Thus $f$ is a well-defined function of $\pi_0(Sing^{\mathbb{A}^1}(X))$.
%\end{proof}
%Since over an algebraically closed field $k$-rational points are dense in a locally finite type scheme over $k$ hence we have,
%\begin{remark} \label{conn}
%The canonical morphism from $\pi_0^{ch}(X)$ to $\pi_0^{\mathbb{A}^1}(X)$ induces a map from $Hom_{Shv(Sm/k)}(\pi_0^{\mathbb{A}^1}(X), \mathbb{A}^1_k)$ to $\mathcal{O}_{ch}(X)$.  
%\end{remark}
\begin{corollary}
If $X \in Sm/k$ is $\mathbb{A}^1$-chain connected, then $\mathcal{O}_{ch}(X)$ is trivial.
\end{corollary}
\begin{remark}
Koras-Russell threefolds of the first kind are $\mathbb{A}^{1}$-chain connected \cite[Example 2.21]{dpo}. Therefore if $X$ is a Koras-Russell threefold of the first kind, then $\mathcal{O}_{ch}(X)$ is trivial. 
\end{remark}
So far we have observed that $\mathcal{O}_{ch}$, the presheaf of $k$-algebras on the category of affine varieties over $k$ is homotopy invariant and it is directly related to $\mathbb{A}^1$-chain-connected component sheaf (Proposition \ref{reg}). However $\mathcal{O}_{ch}$ is not representable in $\mathbf{H}(k)$. 
\begin{lemma} \label{not representable}
$\mathcal{O}_{ch}$ is not representable in $\mathbf{H}(k)$.
\end{lemma}
\begin{proof}
If possible, $\mathcal{O}_{ch}$ is given by $\mathbb{A}^1$-connected component presheaf of some $\mathbb{A}^1$-fibrant object $\mathcal{X}$. Then $\mathcal{O}_{ch}$ satisfies the gluing property by Lemma \ref{glu}, for any elementary distinguished square as in \ref{cd}. Suppose, $X$ is the affine line $\mathbb{A}^1_k$ and the elementary distinguished square [Definition \ref{cd}] is given by Zariski open covering $U=\mathbb{A}^1_k - \{0\}$ and $V=\mathbb{A}^1_k - \{1\}$. Then both $U$ and $V$ contain no affine lines so, $\mathcal{O}_{ch}(U)=\mathcal{O}(U)$ and $\mathcal{O}_{ch}(V)=\mathcal{O}(V)$. So $\mathcal{O}(U) \times_{\mathcal{O}(U \cap V)} \mathcal{O}(V) = \mathcal{O}(\mathbb{A}^1_k)$. But $\mathcal{O}_{ch}(\mathbb{A}^1_k)$ is trivial. So the map $\mathcal{O}_{ch}(\mathbb{A}^1_k)$ to $\mathcal{O}_{ch}(U) \times_{\mathcal{O}_{ch}(U \cap V)} \mathcal{O}_{ch}(V)$ is not surjective. Hence $\mathcal{O}_{ch}$ does not satisfy the gluing property.
\end{proof} 
\begin{remark}
For an affine variety $X \in Sm/k$, let $X_{Nis}$ be the small Nisnevich site. Then $\mathcal{O}_{ch}|_{X_{Nis}}$ is not a sheaf whenever $\mathcal{O}_{ch}(X) \neq \mathcal{O}(X)$ by Lemma \ref{Local Nature}. If $\mathcal{O}_{ch}(X) = \mathcal{O}(X)$, then $\mathcal{O}_{ch}|_{X_{Nis}}$ is a sheaf.
\end{remark}
\begin{question}
Let $X \in Sm/k$ be an affine surface such that $\mathcal{O}(X)$ is a U.F.D. Suppose $\mathcal{O}_{ch}(X)$ is trivial. Is $X \cong \mathbb{A}^2_k$? 
\end{question}
\begin{remark}
If $X$ is a smooth affine surface with $\mathcal{O}(X)$ is a U.F.D. and $\mathcal{O}_{ch}(X)$ is trivial, then $\mathcal{O}(X)^* = k^*$ (by Lemma \ref{constant}, (4)) and by Theorem \ref{ochtrivial}, there is $T \subset X$ dense in $X$ such that for each $x \in T(k)$ there is a non-constant morphism $g: \mathbb{A}^1_k \to X$ such that $x \in Im(g)$. But from this, we cannot conclude that $X$ is dominated by images of $\mathbb{A}^1$ (Definition \ref{several A1}) which ensures the negativity of logarithmic Kodaira dimension of $X$. \par
            However, there are singular affine surfaces with $\mathcal{O}(X)$ is U.F.D and $\mathcal{O}_{ch}(X)$ is trivial. Consider, for instance the singular Pham-Brieskorn surfaces $X_{p,q,r}= \{x^p + y^q+z^r = 0\} \subset \mathbb{A}^3_k$, where $p, q, r \geq 2$ are pairwise relatively prime integers. Then each $X_{p, q, r}$ is factorial \cite[Section 4, Example (c)]{samuel}. On the other hand, since $X_{p, q, r}$ is the affine cone over the closed curve $C_{p, q, r} =  \{x^p + y^q+z^r = 0\}$  in the weighted projective space $\mathbb{P}(m/p, m/q, m/r)$ where $m = lcm(p, q, r)$, it is $\mathbb{A}^1$-chain connected in the naive sense that any two $k$-points of $X_{p, q, r}$ can be connected by finitely many non-constant images of $\mathbb{A}^1_k$, which implies $\mathcal{O}_{ch}(X_{p, q, r}) = k$.
\end{remark}
\begin{remark} \label{iterated ochain}
We have seen $\mathcal{O}_{ch}(X)$ as the regular functions on $\mathbb{A}^1$-chain connected components of $X$ [Proposition \ref{reg}]. We define, $$\mathcal{O}_{ch}^{(n)}(X) := Hom_{Shv(Sm/k)}(\mathcal{S}^n(X), \mathbb{A}^1_k).$$ Then $\mathcal{O}_{ch}^{(n)}(X)$ is a $k$-subalgebra of $\mathcal{O}(X)$ and $\mathcal{O}_{ch}^{(n+1)}(X) \subset \mathcal{O}_{ch}^{(n)}(X)$ (since there is an epimorphism from $\mathcal{S}^n(X)$ to $\mathcal{S}^{n+1}(X)$). Thus inside $\mathcal{O}(X)$, we have a decreasing chain of $k$-subalgebras, not necessarily Noetherian. Does the above chain of $k$-subalgebras stabilize?
\end{remark}
\begin{remark}
Suppose $X \in Sm/k$ is $\mathbb{A}^1$-connected. Then there is an $n$ such that $\mathcal{S}^n(X)$ is trivial. Indeed since $\mathcal{L}(X)$ is trivial, the identity map on $X$ and the constant map given by a $k$-rational point are the same in $\mathcal{S}^n(X)$ for some $n$. Thus any map from $Spec \ \mathcal{O}$ ($\mathcal{O}$ is a Henselian local ring) to $X$ is the same as the constant map in $\mathcal{S}^n(X)$. Therefore the chain $\{\mathcal{O}_{ch}^{(n)}(X)\}_n$ in Remark \ref{iterated ochain} stabilizes for some $n$.
\end{remark}
\begin{question}
Given an affine variety $X$, define $X_i$ inductively as follows: $X_0 = X$ and $X_i = Spec(\mathcal{O}_{ch}(X_{i-1}))$. What is the relation between $X_i$ and the spectrum of $\mathcal{O}_{ch}^{(i)}(X)$? \par
  There is a canonical map from $X_i$ to $X_{i+1}$ for every $i$. Note that if $X$ does not have any non-constant $\mathbb{A}^1_k$, then $\mathcal{O}_{ch}(X) = \mathcal{O}(X)$. Does there exist some $n$ such that $X_n$ has no non-constant $\mathbb{A}^1$?
\end{question}

\section{Appendix} \label{app}

  Denote the category of presheaves of sets on $Sm/k$ by $PSh(Sm/k)$ and the category of presheaves of simplicial sets on $Sm/k$ by $\bigtriangleup^{op} PSh(Sm/k)$. We will call it the category of spaces. The left Bousfield localisation of the global projective model structure \cite[Section 2]{dhi} on $\bigtriangleup^{op} PSh(Sm/k)$ with respect to the Nisnevich hypercovers gives Nisnevich local projective model structure \cite[Section 6]{dhi}.  Following \cite[Section 3.2]{mv}, we consider the left Bousfield localisation of the Nisnevich local projective model structure on $\bigtriangleup^{op} PSh(Sm/k)$ with respect to the class of maps $X \times \mathbb{A}^1 \to X$ for all smooth schemes $X$. The resulting model structure is called the unstable $\mathbb{A}^1$- model structure and the resulting homotopy category is denoted by $\mathbf{H}(k)$. The following are some classes of maps which are invertible in $\mathbf{H}(k)$,  :

\begin{enumerate}
\item A morphism $f : \mathcal{X} \to \mathcal{Y} \in \bigtriangleup^{op} PSh(Sm/k)$  such that 
the induced morphisms on the stalks (for the Nisnevich topology) are weak equivalences of simplicial sets.

\item The projection morphism $pr : \mathcal{X} \times \mathbb{A}^1 \to \mathcal{X}$ for any $\mathcal{X} \in  \bigtriangleup^{op} PSh(Sm/k)$.

\item The structure map of any vector bundle $\mathcal{V} \to X$, for $X \in Sm/k$.
\end{enumerate}

Computation of the unstable motivic invariants of space $\mathcal{X}$  requires a  $\mathbb{A}^1$-fibrant model of $\mathcal{X}$, i.e $i : \mathcal{X} \to Ex^{\mathbb{A}^1}(\mathcal{X})$ such that $i$ is an $\mathbb{A}^1$-weak equivalence and $ Ex^{\mathbb{A}^1}(\mathcal{X})$ Nisnevich fibrant and $\mathbb{A}^1$-local. 
\begin{definition} \label{cd}
An elementary distinguished square (in the Nisnevich topology) is a cartesian square in $Sm/k$ of the form
$$\begin{tikzcd}
U \times_X V \ar[r] \ar[d] &V \ar[d, "p"]\\
U \ar[r,"j"] &X
\end{tikzcd}$$
such that $p$ is an \'etale morphism, $j$ is an open embedding and $p^{-1}(X-U) \to (X-U)$ is an isomorphism (we put the reduced induced structure on the corresponding closed sets) \cite[Definition 3.1.3]{mv}.
\end{definition}
See \cite[Chapter 3]{mv} for more details.
\begin{definition} \label{bg}
A simplicial presheaf $\mathcal{X}$ on $Sm/k$  is said to satisfy the Nisnevich Brown–Gersten property if for any elementary distinguished square in the Nisnevich topology as in Definition \ref{cd}, the induced square of simplicial sets 
$$\begin{tikzcd}
\mathcal{X}(X) \ar[r] \ar[d] &\mathcal{X}(V) \ar[d]\\
\mathcal{X}(U) \ar[r] &\mathcal{X}(U \times_X V)
\end{tikzcd}$$
is homotopy cartesian \cite[Definition 3.1.13]{mv}.
\end{definition}
\begin{remark}
Any fibrant space for the Nisnevich local model structure satisfies the Nisnevich Brown–Gersten property \cite[Remark 3.1.15]{mv}. 
%A space is $\mathbb{A}^1$–fibrant if and only if it is fibrant in the Nisnevich local model structure and it is $\mathbb{A}^1$–local
%\cite[Proposition 2.3.19]{mv}. 
Any $\mathbb{A}^1$-fibrant space satisfies Nisnevich Brown-Gersten property.
\end{remark}

%Nisnevich cd-square and Nisnevich BRown-Gersten property ( See \cite{c} section 3 and the reference from there)

%\begin{remark}
%A Nisnevich fibrant space satisfies Nisnevish Brown Gersten property.
%\end{remark}

The next lemma shows that the connected component presheaf of a Nisnevich fibrant object satisfies the gluing property for any elementary distinguished square.
\begin{lemma} \label{glu}
Suppose $\mathcal{X}$ is a Nisnevich fibrant object. Then for any elementary distinguished square as in \ref{cd}, the induced map $$\pi_0(\mathcal{X}(X)) \to \pi_0(\mathcal{X}(U)) \times_{\pi_0(\mathcal{X}(U \times_X V))} \pi_0(\mathcal{X}(V)) $$ is surjective.
%Cd square and gluing of connected component \cite{c}[Lemma 2.2]. Given the Nisnevich cd square and a Nisnevich fibrant object the connected componenet presheaf satisfies gluing property (not necessarily unique gluing). Formulate this statement.
\end{lemma}
\begin{proof}
\cite[Lemma 2.2]{c}.
\end{proof}

\end{document}